\documentclass[a4paper,12pt, reqno]{amsart}
\usepackage{amsmath,amsfonts,amssymb,amsbsy,amsthm,epsfig,color,mathrsfs}
\usepackage{hyperref}
\newcommand{\ignore}[1]{}

\newcommand{\SL}{\operatorname{SL}}
\newcommand{\Sp}{\operatorname{Sp}}
\newcommand{\SU}{\operatorname{SU}}
\newcommand{\GL}{\operatorname{GL}}
\newcommand{\PGL}{\operatorname{PGL}}
\newcommand{\PSL}{\operatorname{PSL}}

\newcommand{\SO}{\operatorname{SO}}
\newcommand{\bb}{\mathbb}
\newcommand{\Matn}{\operatorname{Mat}}
\newcommand{\C}{\bb C} 

\newcommand{\Z}{\bb Z}

\newcommand{\R}{\bb R}
\newcommand{\N}{\bb N}

\newcommand{\Q}{\mathbb Q}

\newcommand{\cI}{\mathcal{I}}
\newcommand{\cJ}{\mathcal{J}}
\newcommand{\cU}{\mathcal{U}}
\newcommand{\cH}{\mathcal{H}}

\newcommand{\vol}{\operatorname{vol}}

\newcommand\norm[1]{\left\|#1\right\|}

\newcommand\abs[1]{\left|#1\right|}

\newcommand\inn[1]{\left\langle #1 \right\rangle}
\newcommand\set[1]{\left\{{#1}\right\}}

\def\rsemi{\hbox{$\triangleright\!\!\!<$}}

\newtheorem{Theorem}{Theorem}
\newtheorem{Cor}[Theorem]{Corollary}
\newtheorem{Prop}[Theorem]{Proposition}

\newtheorem*{lemma*}{Lemma}

\newtheorem{remark}[Theorem]{Remark}
\newtheorem*{theorem*}{Theorem}
\newtheorem{Def}[Theorem]{Definition}
\numberwithin{equation}{section}
\numberwithin{Theorem}{section}

\begin{document}
\title[Best rates of Diophantine approximation]{Best possible rates of distribution\\  of dense lattice orbits  \\ in homogeneous spaces}
\author{Anish Ghosh, Alexander Gorodnik, and Amos Nevo} 
\address{School of Mathematics, Tata Institute of Fundamental Research, Mumbai, India }
\email{ghosh@math.tifr.res.in}
\address{School of Mathematics, University of Bristol, Bristol UK }
\email{a.gorodnik@bristol.ac.uk}
\address{Department of Mathematics, Technion IIT, Israel}
\email{anevo@tx.technion.ac.il}

\date{\today}
\thanks{The second author acknowledges
  support of ERC. The third author acknowledges support of ISF}

\begin{abstract}
The present paper establishes upper and lower bounds on the speed of approximation in a wide range of natural Diophantine approximation problems. The upper and lower bounds coincide in many cases, giving rise to optimal results in Diophantine approximation which were inaccessible previously. Our approach proceeds by establishing, more generally, upper and lower bounds for the rate of distribution of dense  orbits  of a lattice subgroup $\Gamma $ in a connected Lie (or algebraic) group  $G$, acting on suitable homogeneous spaces $G/H$. The upper bound is derived using a quantitative duality principle for homogeneous spaces, reducing it to a rate of convergence in the mean ergodic theorem for a family of averaging operators supported on $H$ and acting on $G/\Gamma$. In particular, the quality of the upper bound on the rate of distribution we obtain is determined explicitly by the spectrum of $H$ in the automorphic representation on $L^2(\Gamma\setminus G)$.  We show that the rate is best possible when the representation in question is tempered, and show that the latter condition  holds in a wide range of examples.
\end{abstract}

\maketitle

{\small
\tableofcontents
}

\section{Best possible Diophantine exponents}\label{sec:examples}

Our purpose in the present paper is to consider dense orbits of a lattice subgroup $\Gamma\subset G$ acting on a homogeneous space $G/H$, and to give explicit quantitative upper and lower bounds on the rate of approximation of a general point $x_0\in G/H$ by a generic orbit $\Gamma\cdot x\subset G/H$. We will address this problem in considerable generality below and give diverse examples for the approximation exponent under study, but before delving into the general theory, let us point out the following remarkable feature that emerges from  our analysis. It is possible to give an explicit spectral condition which can often be easily verified, and which implies that the upper and lower bounds on the rate of approximation coincide, yielding the best possible result. This gives rise to the solution of an array of  natural problems in classical Diophantine  approximation in real, complex and $p$-adic vector spaces. In order to demonstrate this phenomenon we  begin by presenting some concrete examples where the optimal approximation Diophantine exponent can be obtained.

 Let $F$ denote the fields $\R$, $\C$ or $\Q_p$. Let $G$ be a linear algebraic subgroup
of the group $\SL_n(F)\rsemi F^n$ considered as a group of affine transformations of $F^n$.
We fix a norm on $\R^n$ and $\C^n$, and a (vector space) norm on $\Matn_{n+1}(\R)$ and $\Matn_{n+1}(\C)$. In the local field case we take the standard  valuation on the field, and the standard maximum norm on the linear space $F^n$, and on $\Matn_{n+1}(F)$.  
We view the affine group $\SL_{n}(F)\rsemi F^{n}$, $ n \ge 2$  as a subgroup of $\SL_{n+1}(F)$, specifically as  the stability group of the standard basis vector $e_{n+1}$,  and consider norms on it by restriction from $\SL_{n+1}(F)\subset \Matn_{n+1}(F)$. 
Let $X\subset F^n$ be an affine subvariety which is invariant and homogeneous under the $G$-action,
so that $X\simeq G/H$ where $H$ is closed subgroup of $G$.
We define the distance on $X$ by restricting the norm defined on $F^n$. Let $\Gamma$ be a lattice
subgroup of $G$ which acts ergodically on $G/H$,  so that almost every  $\Gamma$-orbit is dense in $X$.

\begin{Def}\label{def:DE} 
{\rm
Assume that  for $x,x_0\in X$ there exist an exponent $\zeta < \infty$
and a constant $\epsilon_0=\epsilon_0(x,x_0,\zeta) > 0$ such that for all $\epsilon \in (0, \epsilon_0)$, 
the system of inequalities 
$$
\|\gamma^{-1} x- x_0\|\le \epsilon\quad\text{and}\quad\|\gamma\|\le  \epsilon^{-\zeta}.
$$ 
has a solution $\gamma\in \Gamma$.
We define the Diophantine
approximation  exponent $\kappa_\Gamma(x,x_0)$ as the infimum of $\zeta > 0$ such that the foregoing inequalities have 
solutions as stated.
}
\end{Def} 

\begin{enumerate}

\item The exponent above is an analogue for lattice orbits of the uniform approximation studied by Bugeaud and Laurent \cite{BL, BL1, BL-survey} and also generalizes the Diophantine exponent for uniform approximation by $\SL_{2}(\bb Z)$-orbits in $\bb R^2$ considered by Laurent and Nogueira \cite{LN1}. 
\item We remark that the existence of some finite $\zeta$ for which the foregoing Diophantine equation has solutions for arbitrary small $\epsilon$ is a highly non-trivial condition, and is far from obvious in general. It means that $x_0$ has approximations by elements $\gamma x$ in the orbit of $x$ whose norm is bounded by a power of $\epsilon^{-1}$, but a priori the set where this condition is satisfied (for any finite $\zeta$) may be very small. 
\item One can also consider the approximation problem for the system of inequalities
$$
\|\gamma x- x_0\|\le \epsilon\quad\text{and}\quad\|\gamma\|\le  \epsilon^{-\zeta}.
$$ 
Our methods can be used to provide best possible exponents in this case as well. 
\end{enumerate}

Since all norms on $F^n$ and on $\Matn_{n+1}(F)$ are equivalent, it is evident that any choice of norm 
leads to the same exponent.

Notice also that it is evident that the function $\kappa_\Gamma(x,x_0)$ is $\Gamma\times\Gamma$-invariant, and hence when $\Gamma$ acts ergodically on $G/H$, it is almost surely a fixed constant, which we denote by $\kappa_\Gamma (G/H)=k_\Gamma(X)$.

We now proceed to describe some natural examples of classical Diophantine approximation problems 
where the best possible exponent $\kappa_\Gamma(G/H)$ can be computed. The results stated below will all be shown in due course to follow from the general results that will be developed later on, see \S \ref{sec:proofs}.

\subsection{Inhomogeneous Diophantine approximation in the real and the complex plane}

Consider the affine action of the group $\Gamma=\SL_2(\Z)\rsemi \Z^2$  on the real plane $\R^2$.

\begin{Cor}\label{ex:1}
 The approximation exponent for the $\Gamma$-orbits in the plane 
 is given by  $\kappa_{\Gamma}(\R^2)=1$. 
\end{Cor}
More concretely, the problem of  simultaneous inhomogeneous integral Diophantine approximation in the
plane admits the following solution. For any $\eta >0$, for almost every $x_0=(u_0,v_0)$, for almost
every $x=(u,v)$, and for every $\epsilon$ sufficiently small, there are integers $a,b,c,d, m, n$ with
$$
\norm{(au+bv+m, cu+dv+n)-(u_0,v_0)}\le \epsilon
$$ 
such that 
$$
ad-bc=1\quad\hbox{and}\quad 
\max\set{|a|,|b|,|c|,|d|,|m|,|n|} \le  \epsilon^{-1-\eta}.
$$ 

Taking the latter equation modulo integers, we immediately deduce that the rate of approximation by
generic orbits of $\SL_2(\Z)$ in its action on $\mathbb{T}^2=\R^2/\Z^2$ by group automorphisms satisfies
the following. For any $\eta>0$, for almost every $\bar{x}_0\in \mathbb{T}^2$ and for almost every $\bar{x}\in
\mathbb{T}^2$, the equation $\norm{\gamma \bar{x}-\bar{x}_0^\prime}\le  \norm{\gamma}^{-1+\eta} $ has
infinitely many solutions $\gamma\in \SL_2(\Z)$.

Let us note the following :
 \begin{enumerate}
\item The only results we are aware of in the literature regarding estimates of the Diophantine exponent for lattice
  actions on non-compact homogeneous varieties are due to  Laurent--Nogueira \cite{LN1,LN2}. They established that generically $\kappa_\Gamma(x,x_0) \le 
  3$ for $\SL_2(\Z)$ acting {\it linearly} on the plane by explicitly constructing a sequence of approximants using 
a suitable continued fractions algorithm. For the linear action on the real plane, Maucourant and Weiss \cite{MW} have also established an (explicit, but not as sharp) upper bound for $\kappa_\Gamma(x,x_0)$ for arbitrary lattice subgroups of $\SL_2(\R)$ using results on the effective equidistribution of horocycle flows. Diophantine exponents for cocompact lattices acting on the complex plane can be derived from the effective equidistribution theorem of Pollicott \cite{Pollicott}. 
\item We note that for the linear action, determining  the exact value of $\kappa_\Gamma(x,x_0)$ generically remains an open problem, for any lattice subgroup, over any field,  in any dimension. For explicit general upper and lower bounds we refer to \cite{GGN3}. 
\item For the sphere $S^2$ viewed as a compact homogeneous space of $\SO(3,\R)$, an exponent of Diophantine approximation for a suitable lattice in the unit group of a quaternion algebra is a consequence of the celebrated Lubotzky-Phillips-Sarnak construction \cite{lps1}\cite{lps2}. 
Similarly, Diophantine exponents for actions of specially constructed lattices on odd dimensional spheres can be deduced from the sharp spectral estimates due to Clozel \cite{clozel}.  The work of Oh  \cite{heeoh} on spectral estimates for homogeneous spaces of simple compact Lie group can also be used to establish some Diophantine exponent for certain lattices. 
 \end{enumerate}

Let us now consider the problem of simultaneous inhomogeneous Diophantine approximation by pairs of Gaussian integers in $\C^2$, Eisenstein integers, or more generally pairs of algebraic integers in an imaginary quadratic number field. The group $\Gamma=\SL_2(\Z[i])\rsemi \Z[i]^2$ acts ergodically on $\C^2$, and the same holds for 
$\Gamma=\SL_2(\cI)\rsemi \cI^2$, where $\cI$ is the ring of integers of the imaginary quadratic fields $\Q[\sqrt{-D}])$, $D\ge 2$ a positive square free integer. We can now state the following result, whose proof will be given in \S \ref{sec:proofs} below.

\begin{Cor}\label{ex:2}
 Let $\cI$ be the ring of integers in the imaginary quadratic fields $\Q[\sqrt{-1}])$, $\Q[\sqrt{-2}])$,
 $\Q[\sqrt{-3}])$, or $\Q[\sqrt{-7}])$ and $\Gamma=\SL_2(\cI)\rsemi \cI^2$.
Then the approximation exponent for the $\Gamma$-orbits in $\C^2$ is given by  $\kappa_\Gamma(\C^2)= 1 $. 

\end{Cor}
 We note that establishing exact value of the approximation exponent for a general imaginary quadratic field remains an open problem.    

\subsection{Approximation of indefinite ternary quadratic forms}
Consider the variety $\mathcal{Q}_{\sigma,d}(\R)$ of real indefinite ternary quadratic forms of fixed non-zero
discriminant $d$ and signature $\sigma$, on which the group $\SL_3(\R)$ acts 
via $g\cdot Q=Q\circ g^{-1}$. Fixing the standard basis of $\R^3$, each quadratic form can be
represented by a symmetric  $3\times 3$ matrix  $A=A_Q$. The variety $\mathcal{Q}_{\sigma,d}(\R)$
can thus be identified with the set
of $3\times 3$ symmetric matrices with
determinant $d$ and signature $\sigma$.
We use this identification to measure the distances on the variety by the norm difference of the
corresponding representing matrices.

In this setting we have the following Diophantine approximation result. 
\begin{Cor}\label{ex:3}
\begin{enumerate}
\item Given any indefinite ternary form $Q_0\in \mathcal{Q}_{\sigma,d}(\R)$,
for almost every indefinite non-degenerate form $Q\in \mathcal{Q}_{\sigma,d}(\R)$,
 any $\eta > 0$,  and sufficiently small $\epsilon > 0$, there exists $\gamma\in \SL_3(\Z)$ satisfying  
$$
\norm{\gamma^{-1}\cdot Q-Q_0}\le \epsilon\quad\hbox{and}\quad \|\gamma\| \le   \epsilon^{-5-\eta}
$$
and this exponent is the best possible, up to $\eta>0$.  
\item  For any other lattice subgroup $\Gamma$ of $\SL_3(\R)$, the approximation exponent on the variety of indefinite ternary quadratic forms is  again given by $\kappa_\Gamma(\mathcal{Q}_{\sigma,d}(\R))=5$. 
 \end{enumerate}
\end{Cor}


\subsection{Constant determinant variety}
Let $k\neq 0$ and consider the constant determinant variety
 $\mathcal{V}_{n,k}(F)=\set{A\in \Matn_n(F)\,;\, \det A =k}$. The group $G=\SL_n(F)\times \SL_n(F)$ acts transitively on $V$, via $(g,h)A=gAh^{-1}$. The stability group  
 $H\simeq\set{(g,g)\,;\, g\in \SL_n(F)}$, namely the diagonally embedded copy of $\SL_n(F)$. $H$ is the fixed point set of the involution $(g,h)\mapsto (h,g)$ and hence 
the variety $\mathcal{V}_{n,k}(F)=G/H$ in question is a semisimple symmetric space. 

\begin{Cor}\label{ex:4}
\begin{enumerate}
\item Let $\Gamma$ be any irreducible lattice in $\SL_3(F)\times \SL_3(F)$. Then the best possible exponent of Diophantine approximation of $\Gamma$ on the constant determinant variety  is given by $\kappa_\Gamma(\mathcal{V}_{n,k}(F))=4/3$ in the cases $F=\R,\C,\Q_p$. 
\item 
In particular, for any square free integer $d >1$, if $\sigma$ denotes the Galois involution of $\Q[\sqrt{d}]$, and $A\mapsto A^\sigma$ its extension to $\Matn_3(\Q[\sqrt{d}] )$, the action of $\Gamma=\SL_3(\Q[\sqrt{d}])$ on $\SL_3(\R)$ via 
$g\mapsto \gamma g \left(\gamma^\sigma\right)^{-1}$ has exponent exactly $4/3$. 
\end{enumerate}
\end{Cor} 
 
 We note that upper and lower bounds for the exponent of approximation on the constant determinant
 variety when $n > 3$ are established in \cite{GGN3}. However, the exact exponent of approximation in this case 
 remains an open problem.  In the case $n=2$, upper and lower estimates for the exponent are established in the forthcoming work \cite{GGN4}. 

\subsection{Complex structures}
Consider the variety $\mathcal{C}_4(W)$ of complex structures on a four dimensional real vector space $W$. Each complex structure can be identified with a matrix $J\in \Matn_4(\R)$ satisfying $J^2=-I$, and we measure the distance between complex structures by the difference in norm of the representing matrices. The group $\SL_4(\R)$ acts on the space of complex structures on $W$ and  the space $\mathcal{C}_4(W)$ can be identified with $\SL_4(\R)/\SL_2(\C)$. 



\begin{Cor}\label{ex:5}
The action of any lattice subgroup $\Gamma$ of $\SL_4(\R)$ on the space of complex structures on $\C^2\cong \R^4$ has best possible approximation 
exponent given by $\kappa_\Gamma(\mathcal{C}_4(W))=9/4$.
\end{Cor}

\subsection{ Simultaneous Diophantine approximation} 

The action of the order two element $\sigma=\text{ diag }(1,1,-1,-1)$ by conjugation on $\SL_4(F)$ has fixed point group given by 
$$\set{(A,B)\in \GL_2(F)\times \GL_2(F)\,;\, \det (AB)=1}\,.$$
 We consider its semisimple subgroup $H$, namely the kernel of the homomorphism $(A,B)\mapsto \det A$,  and thus the homogeneous space 
$\mathcal{D}_2(F)=\SL_4(F)/\left(\SL_2(F)\times \SL_2(F)\right) $. The latter  can be identified with the variety of direct sum decompositions of $F^4$ into a sum of two $2$-dimensional subspaces endowed with a volume form, namely  
$$(F^4,\vol_4)=(W_1,\vol_2)\oplus (W_2,\vol_2)\,,$$ 
with $\vol_4$ the product of the two volume forms on the subspaces.  $H$ is then the stability group of the decomposition given by (in the obvious notation) $F^4=F^2\oplus F^2$.  The problem of approximation on the variety $G/H$ is that of {\it simultaneous } unimodular Diophantine approximation of two complementary $2$-dimensional subspaces.

More generally, we consider the embedding  of the product group $H=\SL_2(F)^n$ in $G=\SL_{2n}(F)$ in
diagonal blocks, and the homogeneous variety $\mathcal{D}_n(F)=G/H$.  Similarly, approximation on $G/H$ by orbits of a
lattice $\Gamma \subset G$ amounts to simultaneous unimodular Diophantine approximation of $n$
two-dimensional subspaces in $F^{2n}$.

\begin{Cor}\label{ex:6} 
Let $F=\R,\C,\Q_p$, and let $\Gamma$ be any lattice of $\SL_{2n}(F)$, $n\ge 2$. Then the best possible approximation exponent for the lattice orbits on the homogenous space $X$ defined above is 
$\kappa_\Gamma(\mathcal{D}_n(F))=\frac{(2n)^2-1-3n}{2n}$.
\end{Cor}

\subsection{Representations of $\SL_2$}

\subsubsection{Irreducible representations}
Consider the variety $\mathcal{R}_{n}(\R)$ of irreducible linear representations $\tau : \SL_2(\R)\to
\SL_n(\R)$, $n\ge 3$, with two representations $\tau_1$ and $\tau_2$ identified if $\tau_2=\tau_1\circ
j_g$, where $j_g$ denotes conjugation by $g\in \SL_2(\R)$.  The variety can of course be identified with
the set of all irreducible Lie-algebra representations
$\hom(\mathfrak{s}\mathfrak{l}_2,\mathfrak{s}\mathfrak{l}_n)$, two representations being identified if
they differ by precomposition by an element $\text{Ad}(g)$ for some $g\in \SL_2(\R)$. 
In every dimension $n \ge 2$ there is a unique such representation up to equivalence, which we denote by
$\sigma_n$. Thus $\mathcal{R}_{n}(\R)$ is a transitive $\PGL_n(\R)$-space.
Depending on dimension, the variety $\mathcal{R}_{n}$ has at most two connected components $R$,
and the group $G=\SL_n(\R)$ acts transitively on connected components.
Thus each connected component can be identified with $G/H$ where $H\simeq \SL_2(\R)$.

\begin{Cor}\label{ex:7}
For any lattice subgroup $\Gamma \subset \SL_n(\R)$, the exponent of Diophantine approximation for the
action of $\Gamma$ on a connected component $R\subset \mathcal{R}_{n}(\R)$ is $\kappa_\Gamma(R)=\frac12(n^2-4)(n-1)$, provided $n \ge 3$. 
\end{Cor}

\subsubsection{Reducible representations of $\SL_2$}

A reducible representation $\sigma : \SL_2(\R)\to \SL_n(\R)$ can be decomposed to a sum of irreducible
ones, and as we saw above, for a lattice subgroup $\Gamma \subset \SL_n(\R)$ approximation by lattice
orbits  on the variety $\mathcal{R}_\sigma(\R)=\SL_n(\R)/\sigma(\SL_2(\R))$ amounts to a problem in simultaneous unimodular
Diophantine approximation associated with the decomposition, namely simultaneous approximation subject to
the additional constraint of preserving volume forms. Remarkably, in all of these problems the best
possible Diophnatine approximation is achieved. 

\begin{Cor}\label{ex:8}{\bf General representations.}
Let $\sigma : \SL_2(\R)\to \SL_n(\R)$ be any non-trivial representation, $n \ge 3$.
Then for any lattice subgroup $\Gamma \subset \SL_n(\R)$, the exponent of Diophantine approximation for
the action of $\Gamma$ on the variety $\mathcal{R}_\sigma(\R)$ is
$\kappa_\Gamma(\mathcal{R}_\sigma(\R))=\frac12(n^2-4)(d(\sigma)-1)$,
where $d(\sigma)$ is the maximal dimension of an irreducible subrepresentation of $\sigma$. 
\end{Cor}

\subsection{Diagonal embedding and restriction of scalars}

Let $E$ be a totally real field extension of $\Q$ of degree $n-1\ge 2$ with ring of integers
$\mathcal{I}_E$. Embed $\Gamma=\SL_n(\mathcal{I}_E)$ in $\SL_n(\R)^{n-1}=G$ via $\gamma\mapsto
(\sigma_1(\gamma),\dots,\sigma_{n-1}(\gamma))$ where $\sigma_j : E\to \R$ are the $n-1$ distinct field
embeddings. As is well-known (see e.g. \cite[p. 295]{Marg}), $\Gamma$ is an irreducible lattice in
$G$. We denote the diagonally embedded copy of $\SL_n(\R)$ in $G$ by $\Delta(\SL_n(\R))$.
 We can now state : 

\begin{Cor}\label{ex:9} 
$\Gamma$ acts on the variety $X=\SL_n(\R)^{n-1}/\Delta(\SL_n(\R) )$ with best possible exponent, given by $\kappa_\Gamma(X)=(n-2)(n+1)/n$, provided $n \ge 3$. 
\end{Cor}

\subsection{Covering homogeneous spaces}
Finally, let us note that in all of the previous examples concerning the varieties $X\simeq G/H$, 
if $L$ is any non-compact semisimple subgroup of $H$, then the lattice $\Gamma$ acting with the best
possible Diophantine exponent on $X$ also acts with the best possible Diophantine exponent on 
the cover $\tilde X=G/L$. The exponents themselves, however, are generally different. 
We refer to Section \ref{sec:proofs} for the proof of this and the above results.

\subsection{Organisation of the paper.}
In Section \ref{sec:set}, we introduce a general set-up where our arguments apply
 and in Section \ref{sec:gen}  we develop a framework for estimating the approximation exponent
of a lattice group $\Gamma$ acting on homogeneous space $G/H$. In particular, 
we show that this problem reduces to understanding the spectral decomposition of $H$ acting on
$L^2(\Gamma\backslash G)$. In Section \ref{sec:spec}, we discuss the necessary spectral estimates.
Finally, in Section \ref{sec:ex}, we prove the results stated in the introduction.

\section{Definitions, notation and general set-up}\label{sec:set}
\subsection{Very brief overview}
Given a lattice subgroup $\Gamma$ of a locally compact second countable (lcsc)  $\sigma$-compact group $G$, and a closed subgroup $H\subset G$, we consider the case where $\Gamma$ acts ergodically on the homogenous space $G/H$,  with respect to the unique $G$-invariant measure class. 
It then follows that almost every orbit of $\Gamma$ on $G/H$ is dense, and we will develop a quantitative gauge to measure the rate at which denseness is achieved.  An important part of the motivation for this problem is that it includes a broad collection of classical Diophantine approximation problems, as demonstrated in the previous section and will be further demonstrated below.


Our approach is to reduce the study of ergodic properties of lattice orbits on the homogeneous space $G/H$, via the quantitative duality principle developed in \cite{GN3}, to the ergodic properties of the orbits of the stability group $H$ in the space $\Gamma \setminus G$. 

The main tool we will use to study the $H$-orbits in $\Gamma\setminus G$ 
is a quantitative mean ergodic theorem for a family of averages on $H$, acting in  $L^2(\Gamma\setminus G)$. In particular, it is the rate 
of convergence in the mean ergodic theorem for these averages on the stability group $H$ which determines the rate of approximation by lattice orbits on the homogeneous variety $G/H$. Under favorable conditions, it is possible to establish the best possible spectral estimate for these averages, and this makes it possible to establish  the best possible rate of distribution 
of almost every dense orbit.

\subsection{General set-up}\label{sec:gen-set-up}
Continuing with the notation of the previous section,  we fix a discrete lattice subgroup  $\Gamma\subset G$, and $G$ being unimodular, we denote a choice of Haar measure on $G$ by $m_G$. 
Let $H\subset G$ be a closed unimodular subgroup, and choose Haar measure $m_H$ on $H$. It follows that the homogeneous space $G/H=X$ carries a $G$-invariant Radon measure, unique up to multiplication by a positive scalar. We denote by $m_X$ the unique $G$-invariant  measure on $X$ satisfying, for every compactly supported continuous function $f$ on $G$  :
$$\int_G fdm_G(g)=\int_{X}\left(\int_Hf(gh)dm_H(h)\right) dm_X(gH)\,.$$

Let $Y=\Gamma\setminus G$ be the homogeneous space determined by the lattice subgroup, endowed with a
finite $G$-invariant measure. We denote by $m_{Y}=m_{\Gamma\setminus G}$ the unique $G$-invariant measure having the following relation with Haar measure $m_G$. For every compactly supported continuous function $f$ on $G$ :
$$\int_G f(g)dm_G(g)=\int_{\Gamma\setminus G}\left(\sum_{\gamma\in \Gamma}f(\gamma g)\right) dm_{\Gamma\setminus G}(g)\,.$$
Note that this choice of the invariant  measure $m_{Y}$ is not necessarily a probability measure, and its total mass $V(\Gamma)$ is equal to the Haar measure  of a fundamental domain 
of $\Gamma$ in $G$. Thus $m_{Y}(Y)=V(\Gamma)$, and we denote by $\widetilde{m}_{Y}$ the probability measure 
$m_{Y}/V(\Gamma)$.

We fix  a metric $\text{dist} $ on $X$ which satisfies the following regularity property:

\vspace{0.1cm}

\noindent {\bf Assumption 1:} {\it coarse metric regularity.}
 For every compact $\Omega\subset G$ and every compact $S\subset X$, there exists a
 constant $C_0(\Omega,S)>0$ such that for all $g\in \Omega$ and $x,x_0\in S$,
 \begin{equation}\label{eq:bounded-distortion}
 \text{dist} (gx,gx_0)   
\le C_0(\Omega,S)\, \text{dist} (x,x_0)\,.
\end{equation}

We  let $D : G\to \R_+$ be a proper continuous function, and set also $\abs{g}_D=e^{D(g)}$. 
We will assume that  right and left multiplication in $G$ produce only a bounded distortion of $D$,
namely:

\vspace{0.1cm}

\noindent {\bf Assumption 2:} {\it coarse norm regularity.}
For every compact $\Omega\subset G$, there exists a
constant $b(\Omega)>0$ such that for all $u,v\in \Omega$ and $g\in G$,
\begin{equation}\label{eq:distortion} D(ugv)\le D(g)+b(\Omega)\,.\end{equation}
   It then follows 
 that $\abs{ugv}_D\le  e^{b(\Omega)} \abs{g}_D$.

This condition is certainly satisfied if $D$ is coarsely subadditive, 
namely, if  for all 
$g_1,g_2\in G$,
\begin{equation*}D(g_1 g_2)\le D(g_1)+D(g_2)+b\end{equation*}
 with fixed $b>0$. 

We remark that when $F=\R, \C$  and $X\subset F^n$ is a homogeneous space embedded in $F^n$, condition
(\ref{eq:distortion}) will allow us to take any vector-space norm on $\Matn_n(F)$ and restrict it to $G$, not just a submultiplicative norm. This will be convenient in several applications. A similar remark applies to general locally compact fields.

We consider the family of sets $G_t=\{g\in G\,;\, D(g) \le t\}$, which are sets of positive finite Haar
measure on $G$ for sufficiently large $t$.
Their intersection with $H$, namely $H_t=\{h\in H\,;\, D(h) \le t\}$,  are sets of positive
finite Haar measure  in the unimodular subgroup $H$ when $t$ is sufficiently large.

It follows from (\ref{eq:distortion}) that the sets $G_t$ have the following stability property : for
every compact $\Omega\subset G$, there exists a constant $c_1=c_1(\Omega) > 0$ such that
\begin{equation}\label{eq:coarse-admissibility} \Omega G_t \Omega \subset G_{t+c_1}
\end{equation}
for all $t \ge  t_\Omega$.


We now turn to describe the parameters that appear naturally in the estimates of the exponents in the Diophantine approximation problems that will be the main subject of the present paper. 

\subsection{Approximation on homogeneous spaces}\label{sec:definitions}
Fix any point $x_0\in X$, and another point $x\in X$ with the orbit $\Gamma\cdot x$ dense in $X$. We consider the problem of quantitative approximation of the point $x_0\in X$ by points in the orbit $\Gamma\cdot x$, namely establishing an  estimate of the form $d(\gamma^{-1} x,x_0)< \epsilon$ with $\abs{\gamma}_D< B\epsilon^{-\zeta}$, as $\epsilon\to 0$, with a fixed positive $\zeta$. 

Note that this problem is meaningful for any fixed dense orbit $\Gamma\cdot x$, but many possibilities
may arise in the generality discussed here. One is that every orbit of $\Gamma$ is dense in $X$, and
another is that only almost every orbit is dense, but a dense set of points $x\in X$ have non-dense
orbits. There could also be a dense set of points $x^\prime \in X$ where $\Gamma\cdot x^\prime$ is in
fact closed or even finite, although $\Gamma\cdot x$ is dense and $d(x^\prime,\gamma x)$ is arbitrarily small.

Let us now define the exponent of approximation $\kappa$  for dense orbits in the action of $\Gamma$ in its action on (the general metric space) $X$. 

\begin{Def}\label{def:DE} 
{\rm
Given a point $x\in X$ with a dense orbit and a general point $x_0\in X$, assume that there exist  
$\zeta<\infty$ and $\epsilon_0=\epsilon_0(x,x_0,\zeta) > 0$ 
such that for all $\epsilon \in (0, \epsilon_0)$, 
the system of inequalities 
$$
\text{dist}(\gamma^{-1} x, x_0)\le \epsilon\quad\text{and}\quad \abs{\gamma}_D \le  \epsilon^{-\zeta}
$$ 
has a solution $\gamma\in \Gamma$.
Define the Diophantine
approximation  exponent $\kappa_\Gamma(x,x_0)$ as the infimum of $\zeta > 0$ such that the foregoing inequalities have 
solutions as stated.
}
\end{Def}

Let us note the following fundamental, but easily verifiable fact : under the assumptions
(\ref{eq:bounded-distortion}) and (\ref{eq:distortion}) on the gauge $D$, the exponent $\kappa(x,x_0)$ is a {\it $\Gamma\times \Gamma$-invariant function}. Hence when $\Gamma$ is ergodic on $G/H$, it is equal to a constant almost surely, which we will denote by $\kappa_\Gamma(G/H)$.

Our goal is to give explicit upper and lower estimates of the approximation exponent $\kappa_\Gamma(G/H)$, and to that end we introduce the following natural  (and necessary) assumptions, which as we shall see below are satisfied in great generality.

\vspace{0.1cm}

\noindent {\bf Assumption 3:} {\it coarse exponential volume growth.} 
There exist constants  $0 < a^\prime\le a < \infty$ such that for every $\eta > 0$ and a constant
$C_2(\eta)\ge 1$, we have 
\begin{equation}\label{eq:coarse-volume} C_2(\eta)^{-1} e^{t(a^\prime-\eta)}\le m_H(H_t)\le C_2(\eta) e^{t(a+\eta)}\end{equation}
as $t\ge t_\eta$.
Equivalently,
$$
\limsup_{t\to \infty} \frac1t \log m_H(H_t)\le a\quad\hbox{and}\quad\liminf_{t\to \infty} \frac1t \log
m_H(H_t)\ge a^\prime
$$
for finite and positive $a, a'$.  We note that 
because of property (\ref{eq:distortion}) replacing $H$ by a conjugate $gHg^{-1}$  will not affect the values of $a,a^\prime$, and that the constant $C_2(\eta)$ will vary uniformly as $g$ ranges over a compact set. 

Typically, in our examples we will have $a=a^\prime$, and often we will be able to even assert that as $t\to\infty$,
\begin{equation}\label{eq:vol-asyp}
m_H(H_t)=At^re^{at}+o\left(t^{r}e^{at}\right) \text{ for some  $A>0$ and $r\ge  0$.}
\end{equation} This sharper estimate will play a role later on. 

Anticipating some of our later considerations, we note that in \cite{GGN4} we will also allow in (\ref{eq:vol-asyp}) the case where $a=a^\prime=0$, namely where volume growth in $H$ is polynomial. 

\vspace{0.1cm}

\noindent {\bf Assumption 4:} {\it local dimension.} For the family of neighborhoods
$\mathcal{O}_\epsilon(x_0)=\{x\in X\,;\, \text{dist}(x,x_0) < \epsilon\}$, there exist constants $0 <
d^\prime\le d< \infty$ 
such that for every $\eta>0$ and a constant $C_3(x_0,\eta)\ge 1$, we have
\begin{equation}\label{eq:local-dim} C_3(x_0,\eta)^{-1}  \epsilon^{d+\eta}\le m_{X}(\mathcal{O}_\epsilon(x_0) )\le C_3(x_0,\eta)\epsilon^{d^\prime-\eta}
\end{equation}
as $\epsilon \in (0,\epsilon_0(x_0,\eta))$. The constant
$d(x_0)= \limsup_{\epsilon\to 0}\frac{\log(m_{X}(\mathcal{O}_\epsilon(x_0)))}{\log \epsilon}$ 
is called the upper local dimension of $X$ at $x_0$, and the constant
$d^\prime(x_0)= \liminf_{\epsilon\to 0}\frac{\log(m_{X}(\mathcal{O}_\epsilon(x_0)))}{\log \epsilon}$
is called the lower local dimension of $X$ at $x_0$.  Namely, we assume that both 
these dimensions are finite and positive.
Moreover, we note that it follows from property (\ref{eq:bounded-distortion})    
that as $x_0$ varies in compact subset of $X$, the upper and lower local dimension as well as $C_3$ and $\epsilon_0$ above satisfy uniform bounds. 

Clearly, when $X$ is a real connected manifold  $d=d^\prime=\dim_\R (X)$ at every point, and we can also take then $\eta=0$ in (\ref{eq:local-dim}). Typically in the examples we will consider, $d$ and $d^\prime$ will be constant on $X$ and equal.

\vspace{0.1cm}

\noindent {\bf Assumption 5:} {\it  quantitative mean ergodic theorem.}
We  consider the probability $G$-invariant measure $\tilde{m}_{Y}$ on $Y=\Gamma\setminus G$ and define bounded operators 
$\pi_Y(\beta_t)$ on $L^2(Y,\tilde m_Y)$ given by
$$\pi_Y(\beta_t)f(y)=\frac{1}{m_H(H_t)}\int_{H_t}f(yh)dm_H(h)\,\,,\,\, y\in Y\,.$$
We assume that there
exists $\theta > 0$ such that for every $\eta > 0$ there exists a constant $C_4(\eta)>0$ such that
\begin{equation}\label{eq:erg-thm}
\left\|\pi_Y(\beta_t)f-\int_{Y} f \, d\tilde{m}_{Y} \right\|_{L^2(Y,\tilde m_Y)}\le C_4(\eta)
m_H(H_t)^{-\theta +\eta}\|f\|_{L^2(Y,\tilde m_Y)}
\end{equation}
as $t\ge t_\eta$.

We assume that replacing $H$ by a conjugate subgroup $gHg^{-1}$ will not change the validity of this estimate or the value of $\theta$, and that the  constant $C_4(\eta)$ will vary uniformly as $g$ ranges over a compact set. 
We refer to \cite{GN1} for proofs of the ergodic theorem (\ref{eq:erg-thm}) for an extensive class of examples. 

Note that the validity of the quantitative mean ergodic theorem for the averaging operators $\pi_Y(\beta_t)$,  
 implies of course the ergodicity of $H$ on $\Gamma \setminus G$. By the duality principle,  the ergodicity of $\Gamma $ on $G/H$ follows, and in particular, almost every $\Gamma$-orbit in $X=G/H$ is dense.

To study the distribution of  dense lattice orbits in the homogeneous space $X=G/H$, we will need to utilize a Borel measurable section 
$\mathsf{s}_X : G/H \to G$. Letting $\mathsf{p}_X : G \to G/H$ denote the canonical projection, we have $ \mathsf{p}_X\circ \mathsf{s}_X= I_X$, and for each $g\in G$, we have 
$g=\mathsf{s}_X(gH) \mathsf{h}_X(g)$, where $\mathsf{h}_X : G\to H$ is determined by the section $\mathsf{s}_X$, and is Borel measurable. 
Furthermore, note that $\mathsf{h}_X$ is $H$-equivariant on the right, namely $\mathsf{h}_X(gh)=\mathsf{h}_X(g)h$, for $g\in G$ and $h\in H$. 
We assume also that the section is bounded on compact sets in $X$, and that it is continuous in a sufficiently small neighborhood of $x_0$. In our considerations we will assume that $x_0$ is the point in $X=G/H$ with stability group $H=H_{x_0}$, and so that $\mathsf{s}_X([H])=e=\mathsf{s}_X (x_0)$.  

 Note that using the section we obtain a measure-theoretic isomorphism 
between  $\mathsf{p}^{-1}_X(\mathcal{O}_\epsilon(x_0))$ and the direct product $\mathcal{O}_\epsilon(x_0)\times H$, with Haar measure on $G$ taken on the left, and on the right the product of the invariant measure $m_X$ on $X$ (restricted to $\mathcal{O}_\epsilon(x_0))$ and Haar measure $m_H$ on $H$. 

Let $A \subset X$ be a bounded open set, so that $\mathsf{s}_X(A)H_t\subset G$ are bounded sets 
in $G$. Then we claim that 
there exist constants $C'_1=C^\prime_1(A)>0$  and $c'_1=c^\prime_1(A)>0$ such that for all $t > t_{A}$,
\begin{equation}\label{eq:coarse-count}
 \abs{\Gamma\cap
   \left(H_t\mathsf{s}_X(A)^{-1} \right)}\le 
C_1^\prime\, m_H(H_{t+c_1^\prime})\,. 
 \end{equation}
Indeed, taking a sufficiently small neighbourhood $\Omega$ of identity in $G$ such that its
$\Gamma$-translates are disjoint, we conclude that
$$
 \abs{\Gamma\cap
   \left(H_t \mathsf{s}_X(A)^{-1}\right)}\le
\frac{m_G(\left(H_t\mathsf{s}_X(A)^{-1} \right)\Omega)}{m_G(\Omega)},
$$
and it follows from property (\ref{eq:coarse-admissibility}) that 
$$
\left(H_t\mathsf{s}_X(A)^{-1}\right)\Omega\subset H_{t+c_1'} \mathsf{s}_X(A')^{-1}
$$
for some bounded $A'=A'(\Omega,A) \subset X$ and $c_1'=c_1'(\Omega,A)>0$. This implies the estimate (\ref{eq:coarse-count}).

 We note that that replacing $H$ by a conjugate group $gHg^{-1}$ will not affect the validity of this
 estimate, and that the constants  $C_1^\prime(A)$ and $c_1^\prime(A)$  will vary uniformly as $g$ ranges over a compact set. 
 

For a fixed compact set $\Omega^0\subset H$, we denote 
$$\widetilde{\Omega}_\epsilon(x_0)=\mathsf{s}_X(\mathcal{O}_\epsilon(x_0))\cdot \Omega^0\subset G\,,$$
 namely 
$\widetilde{\Omega}_\epsilon(x_0)$ is obtained by lifting the neighborhood $\mathcal{O}_\epsilon(x_0)$ of $x_0\in X$ to $G$ via the section $\mathsf{s}_X$, and then 
multiplying it by a compact set  $\Omega^0\subset H$. Note that for $\epsilon < 1$ (say) the union of  the sets $ \widetilde{\Omega}_\epsilon(x_0)$ is contained in a fixed compact set in $G$, which we denote by $\Omega$.

 We define the associated characteristic function $\chi_\epsilon (g)=\chi_{\widetilde{\Omega}_\epsilon(x_0)}(g)$, which has compact support, and we consider its periodization under $\Gamma$, given by 
$f_\epsilon(\Gamma g)=\sum_{\gamma\in \Gamma}\chi_\epsilon(\gamma g)$. Clearly $f_\epsilon$ is bounded and has compact support in $Y=\Gamma\setminus G$.  

Now note that given a bounded subset $L\subset G$, 
since the intersection of the lattice $\Gamma$ with the bounded set $L\cdot L^{-1}$ is finite and has at most $N(L)$ elements (say), the map $g\mapsto \Gamma g$ is at most $N(L)$-to-$1$ on $L$.  In particular this is valid for  $\widetilde{\Omega}_\epsilon(x_0)$, and it follows that $m_{Y}(\Gamma\widetilde{\Omega}_\epsilon(x_0))\ge N(\Omega)^{-1}m_G(\widetilde{\Omega}_\epsilon(x_0))$. By definition of the measure $m_{Y}$, we also have 
$\int_Y f_\epsilon dm_{Y}= \int_G \chi_\epsilon dm_G= m_G(\widetilde{\Omega}_\epsilon(x_0))$. Furthermore, by its definition the support of the function $f_\epsilon$ in $Y$ is $\Gamma\widetilde{\Omega}_\epsilon(x_0)$, and on its support it is bounded by $N(\Omega_\epsilon(x_0))\le N(\Omega)$.  We can therefore conclude that 
$$\int_Y f^2_\epsilon dm_{Y}\le \int_Y N(\Omega)^2\chi_{\Gamma\widetilde{ \Omega}_\epsilon(x_0)} dm_{Y}\le N(\Omega)^2 m_G(\widetilde{\Omega}_\epsilon(x_0))\,,$$
so that 
\begin{equation}\label{eq:cover}
\norm{f_\epsilon}_{L^2(Y,m_Y)}\le N(\Omega)\sqrt{m_G(\widetilde{\Omega}_\epsilon(x_0))}\,.
\end{equation}

With all the preliminaries in place, we now turn to stating and proving our main approximation results. 

\section{Bounding the exponent approximation for  dense orbits}\label{sec:gen}

\subsection{A lower bound for the approximation exponent}

In the present section 
we will consider a homogeneous space $X=G/H$ and use the bound (\ref{eq:local-dim}) for the lower local
dimension of $X$ and the upper bound (\ref{eq:coarse-volume}) for rate of volume growth of the sets $H_t$
to prove a lower bound on the rate of approximation for almost every point by a dense $\Gamma$-orbit in $X$. 
\begin{Theorem}\label{thm:lower}
Let $G$ be an lcsc group, $H$ a closed subgroup, $\Gamma$ a discrete lattice in $G$ acting ergodically on $X=G/H$. 
Suppose that assumptions 1--4 stated in \S \ref{sec:gen-set-up} and \S\ref{sec:definitions} are
satisfied. In particular, $a$ denotes the volume growth exponent in (\ref{eq:coarse-volume}),
and $d'$ denotes the lower local dimension in (\ref{eq:local-dim}).
Then for $x\in X$ with a dense $\Gamma$-orbit, and for almost every  $x_0\in X$,  the exponent of approximation 
satisfies 
$$\kappa_\Gamma(x,x_0)=\kappa_\Gamma(X)\ge \frac{d^\prime}{a}\,.$$
\end{Theorem}

\begin{proof}
 We fix $x\in X$ with a dense $\Gamma$-orbit, and assume in our discussion below that $H$ is the stability group of $x\in X$.
This amounts to replacing $H$ by conjugate subgroups $g_xHg_x^{-1}$ with some $g_x\in G$, 
for which (\ref{eq:coarse-volume}) is still valid, with constants depending on $x$ and varying uniformly as $x$ varies in a
compact set. 

Given a bounded non-empty open set $A\subset X$, we consider its lift to $G$ via the section $\mathsf{s}_X$, namely the bounded set 
$\mathsf{s}_X(A)$. Note that for $\gamma\in \Gamma$  the conditions $D(\gamma) \le t$ and $\gamma^{-1}
x\in A$  imply that $\gamma\in  H_{t+c_1} \mathsf{s}_X(A)^{-1} $ for some
$c_1=c_1(A)>0$.
Indeed,  $\gamma^{-1}=\mathsf{s}_X(\gamma^{-1} H)\mathsf{h}_X(\gamma^{-1})$ where of course $\mathsf{s}_X(\gamma^{-1} H)\in 
  \mathsf{s}_X(A)$. Furthermore, since $ \mathsf{s}_X(A)$ is a fixed bounded set and
  $\gamma\in G_t$, it follows from (\ref{eq:coarse-admissibility}) that 
$$
\mathsf{h}_X(\gamma^{-1})=\mathsf{s}_X(\gamma^{-1} H)^{-1}\gamma^{-1} \in G^{-1}_{t+c_1}\cap H=H^{-1}_{t+c_1}.
$$ 
Thus $\gamma \in  H_{t+c_1}\mathsf{s}_X(A)^{-1}$.

The upper 
estimate 
we seek for the number of lattice points $\gamma\in \Gamma\cap G_t$ such that $\gamma^{-1} x$ falls in the bounded open set $A$
follow from the upper 
bound (\ref{eq:coarse-count})  on the number of lattice points in $H_{t+ c_1}\mathsf{s}_X(A)^{-1}$, and the upper 
volume growth bound (\ref{eq:coarse-volume}) of $H_t$, as follows:
\begin{align*}
\abs{\set{\gamma\in \Gamma\cap G_t\,;\, \gamma^{-1} x\in A}}
&\le \abs{\Gamma \cap (H_{t+c_1}\mathsf{s}_X(A)^{-1})}\\
&\ll_{A}  m_H(H_{t+c_1+c_1^\prime})\ll_{A,\eta} e^{t(a+\eta)}\,.
\end{align*}
This estimate holds for all $\eta>0$ and sufficiently large $t$.

Now assume that for some $x\in X$ with dense orbit, the set of $x_0$ where $\kappa(x,x_0) < d^\prime/a$ actually has positive measure. Then there exists $\kappa_0 < d^\prime/a$ such that the set $\set{x_0\in X\,;\, \kappa(x,x_0) \le \kappa_0} $ has positive measure. For each point $y$ in the latter set there exists $\epsilon(y)$ such that for all $\epsilon < \epsilon(y)$ there exists some $\gamma\in \Gamma$ with 
$\text{dist}(\gamma^{-1} x,y) \le \epsilon$ and 
$\abs{\gamma}_D \le \epsilon^{-\kappa_0}$. 

For a positive integers $j$, define 
$$A_{j}=\set{y\in X\,;\, \epsilon(y)\le \frac{1}{j}}$$
and then clearly for some $j_0$ the set $A_{j_0}$ has positive measure, and we can consider a bounded subset with positive measure, denoted by
$A^\prime_{j_0}$. 

We now claim that for every $\eta>0$ and sufficiently small $\epsilon>0$,
any bounded set $A$ of positive measure contains a $2\epsilon$-separated set of
cardinality at least $C_\eta\epsilon^{-d^\prime+\eta}$ for some $C_\eta>0$. Indeed, if $\set{y_i\,:\,
  i\in I}$ is a maximal $2\epsilon$-separated subset of $A$, every point in $A$ is at distance less than $2\epsilon$ from one of the points $y_i$.  Then the union over $i\in I$ of the $4\epsilon$-balls $\mathcal{O}_{4\epsilon}(y_i)$ centered at $y_i$ must cover $A$. Hence, the measure $A$ is bounded by 
$$0 < m_X(A)\le \sum_{i\in I}m_X(\mathcal{O}_{4\epsilon}(y_i ))\le \abs{I} \max_{y\in A} C_3(y,\eta)\epsilon^{d^\prime-\eta},$$
where we have used the lower local dimension estimate (\ref{eq:local-dim}) and its uniformity. Thus the size $\abs{I}$ of a maximal separated set satisfies the desired lower bound. 

Let us now apply the last assertion to the bounded positive measure set $A^\prime_{j_0}$, and for each
sufficiently small $\epsilon$ choose a $2\epsilon$-separated subset of cardinality at least
$C_\eta\epsilon^{-d^\prime+\eta}$. Then for all $\epsilon < 1/j_0$, the $\epsilon$-ball centered at each
of these $2\epsilon$-separated points must contain a point of the form $\gamma^{-1} x$ with
$\gamma\in\Gamma$ satisfying $\abs{\gamma}_D \le \epsilon^{-\kappa_0}$. 
These points in the orbit are distinct  
and their total number thus cannot exceed $C'_\eta\left(\epsilon^{-\kappa_0}\right)^{a+\eta}$. 
Thus for all sufficiently small $\epsilon$,
$$C_\eta\epsilon^{-d^\prime+\eta}\le C_\eta'\left(\epsilon^{-\kappa_0}\right)^{a+\eta}.
$$
This implies that $\kappa_0(a+\eta)\ge d'-\eta$ for every $\eta>0$ and hence
$\kappa_0 \ge d^\prime/a$, which contradicts our choice. 
\end{proof}

\begin{remark}
{\rm We note that in order to estimate the exponent for the case when the inequalities are of the form 
$$
\|\gamma x- x_0\|\le \epsilon\quad\text{and}\quad\|\gamma\|\le  \epsilon^{-\zeta},
$$ 
\noindent the same argument can be used. One merely needs to note that the condition $\gamma x \in A$ implies that $\gamma \in \mathsf{s}_X(A)H_{t+c_1}$ and moreover we have the following lattice point estimate, an analogue of $(2.8)$:
$$
\abs{\Gamma \cap (\mathsf{s}_X(A)H_{t})}
\ll_{A}  m_H(H_{t+c_1^\prime}).$$
\noindent The rest of the argument goes through exactly as above.}
\end{remark}

\subsection{An upper bound for the approximation exponent}
We now turn to prove an explicit upper bound on the approximation exponent by points in a dense lattice orbit, complementing the lower bound of the previous section. 

\begin{Theorem}\label{thm:upper}
Let $G$ be an lcsc group, $H$ a closed subgroup, $\Gamma$ a discrete lattice in $G$ acting ergodically on $X=G/H$. 
Suppose that assumptions 1--5 stated in \S \ref{sec:gen-set-up} and \S\ref{sec:definitions} are
satisfied. In particular, $a$ denotes the volume growth exponent in (\ref{eq:coarse-volume}),
$d$ and $d'$ denote the upper and lower local dimensions in (\ref{eq:local-dim}), and $\theta$ denotes the rate in the mean
ergodic theorem in (\ref{eq:erg-thm}).
Then for almost every $x\in X$ and for every $x_0\in X$, the exponent of approximation satisfies
$$
\kappa_\Gamma(x,x_0)=\kappa_\Gamma(X)
\le  \frac{d-d'/2}{a^\prime\theta}.  
$$
\end{Theorem}

\begin{proof}
Let us fix $x_0\in X$, and assume that $H$ is the stability group of $x_0$ in $G$.  We recall from \S
\ref{sec:definitions} the associated section $\mathsf{s}_X : G/H\to G$, and also the sets
$\mathcal{O}_\epsilon(x_0)$ and their lifts, which were chosen as 
$$
\widetilde{\Omega}_\epsilon(x_0)=\mathsf{s}_X(\mathcal{O}_\epsilon(x_0)) \cdot \Omega^0
$$
 with $\Omega^0\subset H$ open and bounded. Finally recall the family of functions $f_\epsilon\in L^2(Y)$ defined there.

Let us apply the averaging operators $\pi_Y(\beta_t)$ acting on $L^2(Y)$ to the function $f_\epsilon$ :
$$\pi_Y(\beta_t)f_\epsilon (\Gamma g)=\frac{1}{m_H(H_t)}\int_{H_t}\sum_{\gamma\in \Gamma}\chi_\epsilon(\gamma^{-1} gh)dm_H(h)\,.$$
To estimate the integrand, let us first determine where it is non-zero. By definition, $\chi_\epsilon(\gamma^{-1} gh)\neq 0$ implies that 
$$
\gamma^{-1} gh\in \widetilde{\Omega}_\epsilon(x_0)\quad \text{or}\quad \gamma\in gh\widetilde{\Omega}_\epsilon(x_0)^{-1}\,\,.$$
Note  that first, if $h\in H_t$ and $g$ ranges over a compact set $L\subset G$,  then $\gamma\in
gH_t  
\widetilde{\Omega}_\epsilon(x_0)^{-1}\subset G_{t+c_1}$ by  (\ref{eq:coarse-admissibility}). 
Second, since $\widetilde{\Omega}_\epsilon(x_0)=\mathsf{s}_X(\mathcal{O}_\epsilon(x_0) )\cdot
\Omega^0$ with $\Omega^0\subset H$,  we conclude that $\gamma^{-1}gH\in \mathcal{O}_\epsilon(x_0)$,
namely $\pi_Y(\beta_t)f_\epsilon (\Gamma g)\neq 0$ implies that $x=gH$ satisfies $\gamma^{-1}x\in
\mathcal{O}_\epsilon(x_0)$  with $\gamma\in \Gamma$ such that $D(\gamma) \le t+c_1$. 

The quantitative mean ergodic theorem (\ref{eq:erg-thm}) gives the estimate
\begin{equation*}
\norm{\pi_Y(\beta_t)f_\epsilon-\int_Y f_\epsilon\, d\tilde m_{Y}}_{L^2(Y,\tilde m_Y)}\le C_4(\eta)
m_H(H_t)^{-\theta +\eta}\norm{f_\epsilon}_{L^2(Y,\tilde m_Y)}
\end{equation*}
for all $\eta>0$ and sufficiently large $t$. We take $\epsilon=e^{-t/\zeta }$ with $\zeta > 0$.
Let us fix a compact set $L \subset G$ and define
$$L^{(t)}=\{g\in L\,;\, \pi_Y(\beta_t)f_{e^{-t/\zeta }}(\Gamma g)=0\}\,.$$
Integrating only on the subset $\Gamma L^{(t)} \subset Y$,  the previous estimate implies 
$$\tilde{m}_{Y}(\Gamma L^{(t)})^{1/2}\cdot \int_{Y} f_{e^{-t/\zeta } }\,d\tilde{m}_{Y}\le
C_4(\eta) m_H(H_t)^{-\theta+\eta}\|f_{e^{-t/\zeta }}\|_{L^2(Y,\tilde m_Y)}.$$

Using the bound (\ref{eq:local-dim}) on the upper local dimension of 
$\mathcal{O}_\epsilon(x_0)$ and the definition of the function $f_\epsilon$, we deduce that
for every $\eta>0$ and sufficiently small $\epsilon$, 
\begin{align*}
\int_Y f_\epsilon\, d\tilde
m_{Y}&=V(\Gamma)^{-1} m_G(\tilde{\Omega}_\epsilon(x_0))\\
&=V(\Gamma)^{-1} m_X(\mathcal{O}_\epsilon(x_0))
m_H(\Omega^0)\gg_{\Omega^0,\eta} \epsilon^{d+\eta}\,.
\end{align*}
Using (\ref{eq:cover}) and the bound (\ref{eq:local-dim}) on the upper local dimension of 
$\mathcal{O}_\epsilon(x_0)$, we also obtain
\begin{align*}
\|f_\epsilon\|_{L^2(Y,\tilde m_Y)} 
\le V(\Gamma)^{-1/2}N(\Omega)\sqrt{m_G(\tilde{\Omega}_\epsilon(x_0))}\ll_{\Omega,\Omega^0,\eta} 
\epsilon^{d^\prime/2-\eta/2}\,.
\end{align*}
Finally, we recall that by (\ref{eq:coarse-volume}) $m_H(H_t)\ge
C_2(\eta)^{-1}e^{(a^\prime -\eta)t}$ for all sufficiently large $t$. 
Putting these estimates together, we have 
\begin{align*}
\tilde m_{Y}(\Gamma L^{(t)})^{1/2}
&\le \left( \int_Y f_{e^{-t/\zeta} }dm_{Y}\right)^{-1} C_4(\eta) m_H(H_t)^{-\theta+\eta}\|f_{e^{-t/\zeta
  }}\|_{L^2(Y,\tilde m_Y)}\\
&\ll_{\Omega,\Omega^0,\eta}\epsilon^{-d-\eta}e^{-(\theta-\eta)(a'-\eta)t}\epsilon^{d^\prime/2-\eta/2}=
e^{-t(-d/\zeta+\theta a^\prime +d'/(2\zeta)-\rho)},
\end{align*}
where $\rho=\rho(\eta)>0$ can be made arbitrary small as $\eta\to 0$. Hence, the exponent can be made
positive provided that $\zeta >   \frac{d-d'/2}{a^\prime \theta}$.
We recall from \S\ref{sec:definitions} that
$$
m_G(L^{(t)})\le N(L) m_{Y}(\Gamma L^{(t)})=N(L)V(\Gamma) \tilde m_{Y}(\Gamma L^{(t)}).
$$
Now it follows from the above estimate that 
 $\sum_{n=1}^\infty m_G(L^{(n)}) < \infty$. Hence by Borel-Cantelli Lemma
almost every point $g\in L$ eventually  avoids $L_{n}$. Equivalently, for almost every $g\in L$ there exists $n(g)$ such that 
$\pi_Y(\beta_n)f_{e^{-n/\zeta }} (\Gamma g)\neq 0$ for all $n \ge n(g)$. As noted above, the non-vanishing  implies that there exists 
$$\gamma\in LH_{n} \tilde{\Omega}_{e^{-n/\zeta }}(x_0)^{-1}\subset G_{ n+c_1}$$
which satisfies $\gamma^{-1}gH\in \mathcal{O}_{e^{-n/\zeta }}(x_0)$. 
Hence, we have proven that there exists $c_1=c_1(L,x_0)>0$, uniform over $x_0$ in bounded sets,
such that  for almost all $g\in L$ and $n > n(g)$, there exists $\gamma\in \Gamma$ satisfying 
$d(\gamma^{-1}gH,x_0) < e^{-n/\zeta}$ with $\gamma$ satisfying $|\gamma|_D\le e^{c_1}e^n\ll
e^{(1+\eta)n}$ with $\eta>0$.  
This proves the statement for $\epsilon=e^{-n/\zeta}$ with $n\in\N$, and for general $\epsilon$
this follows by interpolation.
Finally, by taking $L=U_j$ where $U_j$ is an increasing sequence of subsets exhausting $G$ as $j\to \infty$,
we deduce that the same holds for almost every $g\in G$. 
 This proves that the approximation exponent $\kappa_\Gamma(x,x_0) $ satisfies the upper bound $\kappa_\Gamma(x,x_0) \le \frac{d-d'/2}{a^\prime\theta}$ for every $x_0\in X$ and almost every $x\in X$. 
\end{proof}

\begin{remark}
{\rm In order to study lower bounds for exponents for inequalities of the type
$$
\|\gamma x- x_0\|\le \epsilon\quad\text{and}\quad\|\gamma\|\le  \epsilon^{-\zeta},
$$
\noindent we need to study ergodic averages over $H_t^{-1}$ and therefore we need a mean ergodic theorem for such averages. The same argument then goes through with very minor modifications.
 }

\end{remark}

\subsection{Optimality of the approximation exponent.}  

It is a remarkable fact that when $d=d^\prime$, $a=a^\prime$, and $\theta=1/2$ the spectral upper bound
in Theorem \ref{thm:upper} in fact matches the a priori lower bound on the rate of approximation given in Theorem \ref{thm:lower}, which was derived from more elementary geometric counting arguments.

Retaining the assumptions of  Theorem \ref{thm:upper} and Theorem \ref{thm:lower}, and assuming in addition that $d=d'$ and $a=a^\prime$, we obtain the following sufficient condition for optimality of the rate of distribution of almost every dense lattice orbit on  a homogeneous space. 

\begin{Theorem}\label{thm:optimal}
\begin{enumerate}
\item If the rate of convergence in the mean ergodic theorem  (\ref{eq:erg-thm}) for the averages $\pi_{Y} (\beta_t)$ acting
  in $L^2(\Gamma\setminus G)$ is governed by the square root of the volume of $H_t$ (namely, bounded by
  $C_4(\eta)m_H(H_t)^{-\frac{1}{2}+\eta}$ for every $\eta > 0$), then the approximation exponent is $\kappa_\Gamma(X)=\frac{d}{a}$.
 
\item When $H$ is a connected non-compact semimple algebraic group defined over a local field of characteristic zero, a necessary and sufficient condition for the norm bound in (1) is that the restriction of the unitary representation on $L^2_0(\Gamma\setminus G)$ to $H$ have tempered spherical spectrum as a representation of $H$. 
\end{enumerate}
\end{Theorem} 
 \begin{proof} Part (1) is of course immediate. As to part (2),  when the sets $H_t$ are bi-$K_H$-invariant, where $K_H$ is a maximal compact subgroup of $H$, it suffices that the spherical spectrum of $H$ in its representation on $L^2_0(\Gamma\setminus G)$ is tempered in order to obtain the desired norm bound. 
 Note however that in the problem of approximating a general point $x_0\in X$ by points in a  dense orbit $\Gamma\cdot x$, we can assume 
 that the sets $H_t$ are bi-invariant under a good maximal compact subgroup $K_H$, provided we  choose the original sets $G_t\subset G$ to be  bi-$K$-invariant, namely $KG_t K=G_t$, where $K$ is a suitable maximal compact subgroup of $G$, and $K_H=K\cap H$. 
 \end{proof}

Let us also note the following relation between the pigeon-hole principle, spectral gaps and Diophantine approximation. 

\begin{remark}
{\rm 
\begin{enumerate}

 \item {\it The pigeon hole principle.}   Theorem \ref{thm:lower} implies that for every lcsc group $G$, subgroup $H$ and lattice $\Gamma$ satisfying its assumptions, the best possible upper bound  on the norm of the averaging operators supported on $H_t$ acting in $L^2_0(\Gamma\setminus G)$  is the estimate governed by the square root of the volume. It is quite remarkable that the fact that the spectral decay rate in $L^2_0(\Gamma\setminus G)$ cannot be any faster follows  from a geometric pigeon hole principle for $\Gamma$-orbits in $G/H$.

\item {\it Diophantine approximation and spectral gap.} The previous comment raises the very interesting question of whether a converse statement may hold. To be concrete, let us formulate just two obvious questions, in the case where $G$ and $H$ are connected non-compact almost simple Lie groups, and $\Gamma$ a lattice in $G$.
\begin{itemize}
\item If the exponent of Diophantine approximation of $\Gamma$ on $G/H$  is optimal,  does it follow that the spherical spectrum of $H$ in $L^2_0(\Gamma\setminus G)$ tempered ? 
\item More generally, given a rate of approximation for the $\Gamma$-orbits on $G/H$, is it possible to derive a bound for the  spectral gap of the spherical spectrum of $H$ in $L^2_0(\Gamma\setminus G)$ ? 
\end{itemize} 

\end{enumerate}
}
\end{remark}

\subsection{Sharp approximation by the best possible exponent}
When the best possible approximation exponent $\kappa_\Gamma(x,x_0)=d/a$ is obtained in a given problem, the conclusion
is that $\text{dist}(\gamma^{-1} x,x_0) \le \epsilon$ has solutions $\gamma\in\Gamma$ with $|\gamma|_D\le
B_\eta(x,x_0,\eta)(1/\epsilon)^{(d/a)+\eta}$, for any $\eta > 0$. A considerably sharper statement is that in
fact $\eta =0$ is possible, and  the solutions satisfy 
$|\gamma|_D\le B(x,x_0)(1/\epsilon)^{d/a}(\log (1/\epsilon))^k $ for some fixed $k$. 
This is indeed often the case, and it suffices that the assumptions 3--5 stated in \S2.3 hold in a slightly sharper form. 
Keeping the notation introduced there, assume that
\begin{itemize} 
\item {\bf Assumption 3$^\prime$:} The volume growth on $H$ satisfies 
$$C_2^{-1} t^b e^{at }\le m_H(H_t) \le C_2 t^b e^{at}$$
for all sufficiently large $t$.
\item {\bf Assumption 4$^\prime$:}  The measure of $\epsilon$-neighbourhoods on $X$ satisfies 
$$
C_3^{-1}\epsilon^d\le m_X(\mathcal{O}_\epsilon(x_0))\le C_3\epsilon^d.
$$ 
for sufficiently small $\epsilon$ uniformly over $x_0$ in compact sets.
\item {\bf Assumption 5$^\prime$:} The averaging operators on $\pi_Y(\beta_t)$ satisfy the spectral bound 
\begin{align*}
&\left\|\pi_Y(\beta_t)f-\int_Y f\, d\tilde m_Y\right\|_{L^2(Y,\tilde m_Y)}\\
\le& C_4 m_H(H_t)^{-\theta}\log
(m_H(H_t))^m\|f\|_{L^2(Y,\tilde m_Y)}
\end{align*}
for sufficiently large $t$.
\end{itemize} 

\begin{Theorem}\label{thm:sharp}
Let $G$ be an lcsc group, $H$ a closed subgroup, $\Gamma$ a discrete lattice in $G$ acting ergodically on $X=G/H$. 
Suppose that assumptions 1--2 stated in \S \ref{sec:gen-set-up} and assumptions 3$^\prime$--5$^\prime$ are
satisfied.  Then for every $x_0\in X$ and for almost every $x\in X$, 
the system of inequalities 
$$
\text{\rm dist} (\gamma^{-1} x,x_0) \le \epsilon
\quad\hbox{and}\quad |\gamma|_D\le \epsilon^{-d/2a\theta}(\log (\epsilon^{-1}))^k 
$$
has a solution $\gamma\in \Gamma$ provided that $k > (2m+1-2b\theta)/(\theta a)$ and $\epsilon \in (0, \epsilon_0(x,x_0,k))$.
\end{Theorem}

\begin{proof}
We use the notation and the arguments from the proof of Theorem \ref{thm:upper}, taking advantage of the
superior estimates in the present case. 
We take $\epsilon=t^z e^{-t/\zeta }$ with $z, \zeta > 0$.
We fix a compact set $L \subset G$ and define
$$
L^{(t)}=\{g\in L\,;\, \pi_Y(\beta_t)f_{t^ze^{-t/\zeta }}(\Gamma g)=0\}\,.
$$
Then as in the proof of Theorem \ref{thm:upper} we deduce that
\begin{align*}
\tilde m_{Y}(\Gamma L^{(t)})^{1/2}
&\ll_\eta \left( \int_Y f_{\epsilon} dm_{Y}\right)^{-1} m_H(H_t)^{-\theta}\log
(m_H(H_t))^m\|f_{\epsilon}\|_{L^2(Y,\tilde m_Y)}\\
&\ll_{\Omega,\Omega^0,\eta}\epsilon^{-d}(t^be^{ a t})^{-\theta} t^m\epsilon^{d/2}
=t^{m-dz/2-\theta}e^{dt/(2\zeta)-a\theta t}.
\end{align*}
We choose $\zeta=d/(2\theta a)$ and $z>(2m+1-2b\theta)/d$.
\begin{align*}
m_G(L^{(t)})\ll \tilde m_{Y}(\Gamma L^{(t)})\ll t^{-1-\eta}
\end{align*}
for some $\eta>0$. 
Hence, it follows that $\sum_{n=1}^\infty m_G(L_{n}) < \infty$.
Now we may argue as in the proof of Theorem \ref{thm:upper} to conclude that
for almost every $g\in L$, there exists $n=n(g)$ such that for all $n>n(g)$ there exists
$\gamma\in \Gamma$ such that $d(\gamma^{-1}gH,x_0)\le n^ze^{-n/\zeta}$ and $D(\gamma)\le n+c_1$.
Equivalently, for sufficiently small $\epsilon$ of the form $\epsilon=n^ze^{-n/\zeta}$,
there exists
$\gamma\in \Gamma$ such that $d(\gamma^{-1}gH,x_0)\le \epsilon$ and $|\gamma|_D\le e^{c_1}\phi^{-1}(\epsilon)$
where $\phi(u)=(\log u)^z u^{-1/\zeta}$. We note that $\phi^{-1}(\epsilon)\ll \epsilon^{-\zeta}\log (\epsilon^{-1})^{\zeta z}$
for all sufficiently small $\epsilon$. Hence, this implies the required estimate for almost all $x=gH$ with
$g\in L$. Finally, the proof of the theorem can be completed as before.
\end{proof} 

\begin{remark}\label{rem:sharp}  
{\rm
Let us note that assumption 4$^\prime$ is of course immediate when $G$
  is a real or $p$-adic Lie group. Assumption 3$^\prime$ holds in considerable generality for closed subgroups of Lie and algebraic groups, and in particular for all semisimple algebraic subgroups, see \cite{GW} and \cite[Ch. 7]{GN1}. Assumption 5$^\prime$ also holds quite generally, and we note than when $H$ is a semisimple algebraic subgroup, the best possible spectral estimate $\theta=\frac12$ is equivalent to showing that the representation $\pi_{Y}^0$ restricted to $H$ is weakly contained in a multiple of the regular representation of $H$. But then the spectral estimate of  $\pi_{Y}^0(\beta_t)$ is via the Harish Chandra $\Xi$-function of $H$. The well-known estimates of Harish Chandra  \cite{HC1, HC2, HC3} then imply that the estimate given in assumption 5$^\prime$ is indeed valid.  
These facts will be elaborated further below. 
}
\end{remark}
We now turn to to describe spectral methods that yield best possible Diophantine rates. We note that all the considerations we develop below are required for the proofs of the Diophantine exponents stated in our examples above.  It is more useful and efficient, however, to state the spectral estimates in general form, and we will devote \S 5 to a detailed account of their applications.  

We illustrate Theorem \ref{thm:sharp} on the example of the affine action of $\Gamma=\SL_2(\Z)\rsemi
\Z^2$ on the plane $\R^2$. In this case, we deduce that for every $x_0\in \R^2$ and almost every
$x\in\R^2$, the system of inequalities 
$$
\|\gamma x-x_0\|\le \epsilon\quad\hbox{and}\quad \|\gamma\|\le \epsilon
^{-1}\log(\epsilon^{-1})^{3+\eta}
$$
has a solution $\gamma\in \Gamma$ for all $\eta>0$ and sufficiently small $\epsilon$.
This improves Corollary \ref{ex:1}. We note that the other results
stated in the introduction also admit such logarithmic improvements.

\subsection{Pointwise bounds on approximation exponents}
In Theorem \ref{thm:lower} we have fixed any starting point $x\in X$ with dense orbit $\Gamma\cdot x$, and proved that for {\it almost every } point $x_0$ we have $\kappa(x,x_0) \ge \frac{d^\prime}{a}$. On the other hand, in Theorem \ref{thm:upper} we have fixed any point $x_0\in X$, and showed that for {\it almost every} point $x$ we have $\kappa(x,x_0) \le \frac{d-d'/2}{a^\prime\theta}$. Consequently, the conclusion we can draw 
from their conjunction is that  
$$
\frac{d^\prime}{a}\le \kappa(x,x_0)\le  \frac{d-d'/2}{a^\prime\theta}
$$
for {\it almost every} $x,x_0\in X$. 

It is a natural problem to investigate further the set where the upper or the lower estimate hold, and the set where both are valid.  We will consider this problem only in the following important special case. If the action of $\Gamma$ on $X$ is isometric with respect to the metric $d$, then $d(\gamma^{-1} x,x_0)=d(x,\gamma x_0)$.  Furthermore, in an isometric action, if one orbit is dense then all of them are. The argument used to prove Theorem \ref{thm:upper}  therefore implies the following. 

\begin{Theorem}\label{thm:isometric}
Assume, in addition to the hypotheses of Theorem \ref{thm:upper}, that the action of $\Gamma$ on $X=G/H$
is isometric. Then {\emph every orbit} approximates fast almost every point with the exponent
$\kappa=\frac{d-d'/2}{a^\prime\theta}$, namely  for every point
$x_0\in X$, and for almost every point $x\in X$, there exists $\gamma\in \Gamma$ with 
$$
d(\gamma x_0,x)\le \epsilon\quad\hbox{ and }\quad \abs{\gamma}_D \le \epsilon^{-\zeta}
$$
provided that $\zeta > \kappa=\frac{d-d'/2}{a^\prime\theta}$ and $\epsilon$ is sufficiently small.  
\end{Theorem}
It is conceivable that the conclusion holds, in the isometric case, for every single pair $(x,x_0)$, but this remains an open problem. 


{\it The problem of uniformity on a co-null set of orbits. } Theorem \ref{thm:upper}  establishes that {\it every  point} $x_0\in X$  can be approximated fast, namely with  approximation exponent is $\kappa$, by {\it almost every orbit} $\Gamma\cdot x$. Another natural open problem is whether this conull set can be taken to be one and the same for every $x_0$, for example, whether it can be defined by suitable specific Diophantine conditions. 
  
  We note that in \cite{LN1} and \cite{LN2} Laurent and Nogueira consider the case of $\SL_2(\Z)$ acting on $\R^2$, and assuming specific Diophantine conditions on the point $x$,  they deduce that the orbit $\SL_2(\Z)\cdot x$ approximates every point. In particular, they show that there is a fixed co-null set of points $x\in \R^2$ such the orbit of each of them under $\SL_2(\Z)$ approximates every point (at a fixed rate).


\section{Spectral estimates}\label{sec:spec}

\subsection{Subgroup temperedness problem} 
To apply Theorem \ref{thm:optimal} to obtain optimal approximation results for $\Gamma$-orbits in $G/H$, we must  address the following basic questions.  Consider an  lcsc group $G$,  a closed unimodular subgroup $H$, a unitary representation $\pi : G\to \mathcal{U}(\mathcal{H})$, and a lattice subgroup $\Gamma \subset G$.  Let $\lambda_H$ denote the regular representation of $H$ on $L^2(H)$, and let $\pi\le_w \pi^\prime$ denote weak containment of unitary representations.  Recall that the unitary representation $\pi$ of $G$ is {\it weakly
contained} \index{weak containment} in the unitary representation $\pi^\prime$ if for every $F\in
L^1(G)$ the estimate $\norm{\pi(F)}\le \norm{\pi^\prime(F)}$ holds. Clearly, if
$\pi$ is strongly contained in $\pi^\prime$ (namely equivalent to a
subrepresentation), then it is weakly contained in $\pi$. 
\begin{Def}\label{tempered-rep}
 \begin{enumerate}
 \item  $(G,H,\pi)$ is a {\it tempered triple} if the restriction of the representation $\pi$ to $H$ is weakly contained in a multiple of the regular representation of $H$, denoted $\infty \cdot \lambda_H$. 
\item  $H\subset G$ is $(\Gamma\setminus G)$-tempered if  the restriction to $H$ of the representation $\pi^0_{\Gamma\setminus G}$ of $G$ on  $L^2_0(\Gamma\setminus G)$ is weakly contained in $\infty \cdot \lambda_H$, namely if $H$ is $(G,H,\pi^0_{\Gamma\setminus G})$-tempered. 
\item $H$ is a tempered subgroup of $G$ if the restriction of every  unitary representation of $G$ without invariant unit vectors to $H$ is weakly contained in $\infty \cdot \lambda_H$. 
\end{enumerate}
\end{Def}

 Given a compact subgroup $K\subset G$, we will also consider when the restriction to $H$ of the spherical spectrum of the representation $\pi^0_{\Gamma\setminus G} $, or any representation $\pi$ without invariant unit vectors, is weakly contained in $ \lambda_H$. Here a spherical irreducible representation of $G$ is one containing a $K$-invariant unit vector. In these situations we will say that $H$ is $\Gamma\setminus G$-spherically tempered or that $H$ is a spherically-tempered subgroup of $G$.
 
 Let us note that the case of $H=G$, namely when the representation of $G$ in $L^2_0(\Gamma\setminus G)$ is a tempered  or spherically tempered representation of $G$, is a major open problem in the theory of automorphic forms. Thus for $G=\SL_2(\R)$ and $\Gamma\subset \SL_2(\Z)$ the congruence subgroups, temperedness amounts to the celebrated Selberg conjecture. For further discussion we refer to  \cite{BLS,BS}, the surveys \cite{Sarnak} and \cite{BB} and the references therein. 
 
It is a remarkable and most useful fact that the temperedness problem for  subgroups $H\subsetneq G$ exhibits  very different behavior : it is a ubiquitous phenomenon which can be established using several different methods.  The rest of our discussion will be devoted to demonstrating this fact and discussing further examples where the best possible spectral estimate of $H$ on $L^2_0(\Gamma\setminus G)$ is obtained. 
 
Before proceeding, let us first note the following. 
\begin{remark}
{\rm 
\begin{enumerate}
\item The properties stated in Definition \ref{tempered-rep} are interesting as stated only for non-amenable subgroups $H$. When $H$ is amenable, typically the representation of $H$ on $L^2_0(\Gamma\setminus G)$ 
will be weakly equivalent to the regular representation of $H$. But then the operator norm of the averaging operators $\pi_{H}(\beta_t)$ 
in $L^2(H)$ is identically $1$, owing to the existence of an asymptotically invariant sequence of unit vectors in $L^2(H)$. Thus $\norm{\pi^0_{\Gamma\setminus G}(\beta_t)}=1$ for all $t$, and the operator norm does not decay at all. 
\item Nevertheless, it is an important and useful fact that when considering amenable groups, it is often the case that the relevant operators have norm in Sobolev spaces which converge to zero.  The rate of decay  can be estimated using again  the spectral characteristics of the representation of $G$ on $L^2_0(\Gamma\setminus G)$. Such estimates will appear in detail in \cite{GGN4}. 

\item The definition of tempered subgroup above  is different than the notion of a subgroup $H\subset G$
  being $(G,K)$-tempered introduced by Margulis \cite{Ma}. The latter is strictly stronger than
  Definition \ref{tempered-rep}, and it  deviates from standard terminology, in which a representation of
  a semi simple group is spherically tempered if (a dense subspace of) its spherical matrix coefficients
  are in $L^{2+\eta}(G)$ for every $\eta>0$, and not necessarily in $L^1(G)$. We will therefore use the term ``$(G, K)$-tempered in $L^1$" to refer to this situation. As we shall see below in Theorem \ref{thm:L1L2}, if a closed subgroup $H$ of a semisimple group $G$ is $(G,K)$-tempered in $L^1$, then $H$ is in fact a tempered subgroup of $G$, in the sense defined in the present paper. 

\end{enumerate}
}
\end{remark}
\subsection{Ubiquity of temperedness}
  To show that the phenomenon of subgroup temperedness is remarkably prevalent and robust, we will very briefly present five general methods establishing it, and then proceed to explain how to utilize them and give a considerable number of further examples. 
    Specifically, for given $G,H $ and $\Gamma$ the representation of $H\subset G $ in $L^2_0(G/\Gamma)$ can be shown to be tempered or spherically tempered using the following methods. 
  \begin{enumerate}
  \item {\bf Kazhdan's original argument and some variants.}  
  Here $G$ is a simple $S$-algebraic group (for instance, but many other lcsc groups qualify), and $H\cong \SL_2(F)\subset \SL_2(F) \rsemi_\sigma F^k$ is embedded in $G$ as a closed subgroup, where $\sigma$ is a representation of $\SL_2(F)$ on $F^k$ without non-zero invariant vectors. The lattice subgroup $\Gamma\subset G$ is arbitrary. 
     \item {\bf Tensor power argument.}  Here $G$ is a semisimple $S$-algebraic group $G$ (or a subgroup of $G$) embedded diagonally in $G^{n(G)}$, where $n(G)$ is an integer determined explicitly by  spectral estimates which give an explicit form of property $T$. Again the lattice is arbitrary. 
     \item {\bf Restriction of universal pointwise bounds.} $K$-finite matrix coefficients on a
       semisimple $S$-algebraic group $G$ satisfy several universal pointwise bounds. It is often
       possible to establish directly that their restriction (or the restriction of some of them) to a
       subgroup $H$ is in $L^{2+\eta}(H)$ for every $\eta>0$. This includes the case of $L^1$-tempered $(G, K)$ representations. 
  Once again the lattice here is arbitrary. 
    \item {\bf Direct proofs of temperedness of $G/\Gamma$.} For some specific choices of the group $G$ and the  lattice $\Gamma$ temperedness has been established directly. This applies for example to  $\SL_2(\Z)\subset \SL_2(\R)$ and $\SL_2(\Z[i])\subset \SL_2(\C)$, as well as some of their congruence subgroups. 
    \item {\bf Bounds towards Ramanujan-Selberg conjecture}. If $G$ is of real rank one, for example $\SO(n,1)$, some bounds towards Ramanujan-Selberg conjecture are known for some lattices.  The restriction of $\pi^0_{\Gamma\setminus G}$  to $\SO(k,1)$ can be spherically-tempered or close to tempered for certain $1 \le k < n$. 
  \end{enumerate}
  
We now turn to presenting and utilizing these arguments in more detail, together with some other relevant spectral estimates. We remark that in all the examples given below of tempered triples $(G,H,\Gamma)$ the assumptions stated in 
\S \ref{sec:definitions} are satisfied, and best possible approximation rates for the action of $\Gamma\subset G$ on the homogeneous varieties $G/H$ are obtained. For more on these arguments, we refer to \S \ref{sec:proofs}.

\subsection{Restriction to $\SL_2$-subgroups : Kazhdan's argument}

The phenomenon of subgroup temperedness appeared already in Kazhdan's original proof of property $T$ in \cite{K}. We quote the basic observation in the following form which will be convenient for our purposes, and which summarizes the results stated by Zimmer in \cite[Ex. 7.3.4., Thm. 7.3.9]{Zi} and Margulis \cite[Ch. III, Prop. 4.5]{Ma}. Note also that a remarkable geometric proof for the case of $\SL_2(\R)\rsemi \R^2$ is given in \cite[Prop. 3.3.1]{HT}. 
\begin{Theorem}\label{thm:sl2 restriction} 
Let $F$ denote a locally compact non-discrete field of characteristic zero. 
Let $\sigma : \SL_2(F)\to \SL_k(F)$ be a linear  representation without invariant non-zero vectors, $k \ge 2$.  Let $S=\SL_2(F)\rsemi F^k$ be the semi-direct product group. Then for any non-trivial unitary representation $\pi : S\to  \mathcal{U}(\mathcal{H}) $ on a Hilbert space $\mathcal{H}$, if the representation $\pi$ does not admit a $\pi(F^k)$-invariant unit vector, then the representation  $\pi$ restricted to $\SL_2(F)$ is tempered. 
\end{Theorem}

Let now $G$ be any lcsc group,  let $S$ be as in Theorem \ref{thm:sl2 restriction} and  let $\sigma : S\to G$ be a  representation with closed image. Clearly the image  of $F^k$ cannot be compact without being trivial, since the image of $\SL_2(F)$ normalizes it, but does not centralize it. Therefore, if
 $\pi$ is any mixing unitary representation of $G$, namely a representation whose matrix coefficients vanish at infinity, then $\pi\circ \sigma(F^k)$ has no invariant unit vectors. Therefore the triple $(G, \sigma(\SL_2(F)), \pi)$ is tempered. 

Note that when $G$ is the group of $F$-points of an algebraic group over $F$, $H=\sigma(\SL_2(F))$ acts via the adjoint representation on $\mathfrak{g}$, which decomposes as a sum of irreducible linear representations of $H$. Let $\mathfrak{v}\subset \mathfrak{g}$ be an $H$-invariant  irreducible subspace of dimension at least $2$ which forms an abelian subalgebra $\mathfrak{v}$.  Clearly the group $V=\exp \mathfrak{v}$ cannot be compact, since it is normalized by $H$ but not centralized by it.  Then $S=H\rsemi V$ satisfies the hypothesis of Theorem \ref{thm:sl2 restriction}. 
(Note that when we assume that the adjoint representation of $H$ on $\mathfrak{g}$ has an irreducible subspace of dimension at least $2$ which is realized on an Abelian subalgebra, then the exception case giving rise to the oscillator representation of $\Sp_{2n}$ does not arise). 

When $G$ is a semisimple group and the lattice $\Gamma \subset G$ is irreducible,  the Howe-Moore mixing theorem for irreducible actions \cite{HM} implies that the representations $\pi^0_{\Gamma\setminus G}$ is mixing. The Abelian group $V=\exp (\mathfrak{v}) $  is non compact and hence acts ergodically on $\Gamma \setminus G$. We summarize these considerations in the following result which is  corollary of Theorem \ref{thm:sl2 restriction}. 

\begin{Theorem}\label{thm:sl2}
Let $F$ be a locally compact non-discrete field of characteristic zero, and let $G$ denote the $F$-points of a connected semisimple linear algebraic group defined over $F$. 
Let $\sigma : \SL_2(F)\to G$ be a non-trivial rational representation, and assume that the adjoint representation of 
$H=\sigma(\SL_2(F))$ on the Lie algebra $\mathfrak{g}$  admits an irreducible subrepresentation of dimension at least $2$ on an Abelian subalgebra. 
If $\Gamma\subset G$ is an irreducible lattice subgroup, then the unitary representation of $H$ on $L^2_0(\Gamma\setminus  G)$ is tempered. 
\end{Theorem} 

Let us give some explicit examples of subgroup temperedness deriving from the results above.
  \begin{enumerate}
\item 
$H=\SL_2(F)\subset \SL_3(F)=G$ is a tempered subgroup, provided the representation of $\SL_2(F)$ in $\SL_3(F)$ is non-trivial and reducible. 
\item $H=\SL_2(F)\subset \SL_n(F)=G$, $n \ge 3$ is a tempered subgroup, for a host of reducible  representations of $\SL_2(F)$ in $\SL_n(F)$, for example when the subgroup preserves the direct sum decomposition $F^n=F^2\oplus F^{n-2}$ and acts irreducible on the first summand. 
\item $H=\SL_2(F)\times \cdots \times \SL_2(F)\subset \SL_{2n}(F)=G$ is a tempered subgroup,  where the $n$-fold product $\SL_2(F)\times\cdots \times  \SL_2(F)$ is embedded so as to preserves the direct sum decomposition $F^{2n}=F^2\oplus\cdots \oplus F^2$.  
\end{enumerate}

\subsection{The tensor power argument}
Recall that a unitary representation $\pi$ of an lcsc group $G$  is called a {\it strongly $L^p$-representation} \index{lp-representation@$L^p$-representation}
if there exists a dense subspace 
$\cJ\subset \cH$, such that the matrix coefficients
$\inn{\pi(g)v,w}$ belong to $L^p(G)$, for $v,w\in \cJ$.  If $\pi$ in an $L^{p+\eta}$-representation for all $\eta > 0$, we will call it an $L^{p^+}$-representation.  
 When $K\subset G$ is a compact subgroup, recall also that a vector $v\in \cH$
is called {\it $K$-finite}  (or $\pi(K)$-finite) if its orbit under $\pi(K)$
spans a finite dimensional space.  

We recall the following spectral estimates, which will play an 
important role below. They are due to \cite[Thm. 1]{chh}, \cite{Co}, \cite{H} (see \cite{HT} for a simple proof). 
\begin{Theorem}\label{tensor}
{\bf Tensor power argument.} 
For any lcsc group $G$ and any strongly continuous unitary representation $\pi $
\begin{enumerate}
\item 
If $\pi$ is strongly $L^{2^+}$, then $\pi$ is
weakly contained in the regular representation $\lambda_G$.  
\item 
If $\pi$ is strongly $L^p$, and $n$ is an integer satisfying 
$n \ge p/2$, then 
$\pi^{\otimes n}$ is strongly contained in $\infty\cdot \lambda_G$.  

\end{enumerate}
\end{Theorem}

We now state the following result, 
which summarizes a number of results 
due to \cite{Co,HM,BW,li,oh2}
 in a form convenient for our purposes.

\begin{Theorem}\label{Lp representations}
{\bf $L^p$-representations. }

Let $F$ be an locally compact 
non-discrete field of characteristic zero. 
Let $G$ denote the group of $F$-rational 
points of an algebraically connected 
semisimple algebraic group which is almost $F$-simple.
Let $\pi$ be a unitary representation of $G$
 without non-trivial finite-dimensional 
$G$-invariant subspaces (or equivalently 
without $G^+$-invariant unit vectors).\footnote{$G^+$ denotes the closed subgroup of $G$ generated by
  one-parameter unipotent subgroups.}
\begin{enumerate}
\item   
If the $F$-rank of $G$ is at least $2$ 
then $\pi$ 
is strongly $L^{p^+}$ 
with $p^+=p^+(G) < \infty$ depending only on $G$. 
\item If the $F$-rank of $G$ is $1$, then 
any unitary representation $\pi$ admitting a spectral 
gap (equivalently, which does not contain an 
asymptotically invariant 
sequence of unit vectors) is strongly $L^{p^+}$ for 
some $p^+=p^+(\pi)< \infty$. In particular, every 
irreducible infinite-dimensional representation 
has this property. 
\item When $G$ is $S$-algebraic group and $\pi$ is such that 
  every simple component has a spectral gap and no non-trivial finite-dimensional subrepresentations,  then $\pi$ is a strongly  $L^{p^+}$-representation. 
\end{enumerate}
\end{Theorem}

We remark that an explicit estimate of the best exponent $p^+(G)$ valid for all unitary representation as above when $G$ is simple and has $F$-split rank at least two is given by \cite{H}\cite{Sca} for $\SL_n(F)$ and $\Sp_{2n}(F)$, by  \cite{li}\cite{LZ} in the Archimedean case, and by \cite{oh2} for general $p$-adic groups.  

Let us define $n(G)$ as the least  integer greater or equal to $p^+(G)/2$. 


We can now state the following result which establishes that diagonal embeddings in product groups are often tempered. 

\begin{Theorem}\label{tensor power} Assume that $L$ is an $F$-simple algebraic group defined over $F$ with $F$-rank at least two, as in Theorem \ref{Lp representations}, and assume that $L=L^+$. 
Let $H=\Delta(L)$ be the diagonal embedding of $L$ in $L^{n(L)}=G$. If $\Gamma \subset G$ is an irreducible lattice, then the restriction of the representation $\pi^0_{\Gamma\setminus G}$ of $G$ on $L^2_0(\Gamma\setminus G)$ to $H$ is a tempered representation of $H$. 
\end{Theorem}

\begin{proof} Choose any $n=n(L)$ unitary representations $\pi_1,\dots \pi_n$ of $L$ which are irreducible and infinite dimensional, and consider the representation $\tau=\pi_1\otimes \pi_2\otimes \cdots \otimes \pi_n$ as a representation of $G=L^n$. Now restrict $\tau$ to the subgroup $\Delta(L)=H\subset G=L^n$, where $\Delta(L)$ is the diagonally embedded subgroup in $L^n$. The restriction of $\tau$ to $H$ has a dense subspace of matrix coefficients which are in $L^{2+\eta}(H)$, for every $\eta > 0$. Hence the restriction of the representation $\tau$ to the subgroup $H$ is weakly contained in the regular representation $\lambda_H$ of $H$, by Theorem \ref{tensor}(1).  

Let us apply this fact to the representation of $G$ on $L^2_0(\Gamma\setminus G)$, where $\Gamma\subset L^n=G$ is an irreducible lattice. 
The representation of $G$ on $L^2_0(\Gamma\setminus G)$ decomposes to a direct integral of irreducible unitary representations of $G$ of the form $\pi_1\otimes \pi_2\otimes \cdots \otimes \pi_n=\tau $, with each $\pi_i$ an irreducible infinite dimensional unitary representation of $L=L^+$. This follows since the lattice $\Gamma$ is irreducible, and so each of the direct summands in the direct sum $L^n$ is mixing on $\Gamma \setminus G$. Hence the restriction of each such representation $\tau$ to $H$ is a tempered representation of $H$. Since the irreducible constituents $\tau$ in the direct integral decomposition of $\pi^0_{\Gamma\setminus G}$ are tempered when restricted to the subgroup $H$, it follows that their direct integral, the representation on $L^2_0(\Gamma \setminus G)$ is also tempered when restricted to $H$.  
\end{proof}
Some concrete examples of tempered subgroup that arise are as follows.

\begin{enumerate} 
\item 
Let $H=\Delta(\SL_3(F))\subset \SL_3(F)\times \SL_3(F)=G$ with $\Delta(\SL_3(F))$ being the diagonally embedded copy of $L=\SL_3(F)$. Then $H$ is a $\pi_{\Gamma\setminus G}^0$-tempered subgroup of $G$. 
Indeed $\SL_3(F)$ has property $T$, with  $n({\SL_3(F)})=2$, so that Theorem \ref{tensor power} applies.  

\item More generally, if $H$ is simple and $H\subset \Delta(\SL_n(F))\subset \SL_n(F)^{k}=G$, where $n \ge 3$ and $k \ge n-1$, then $H$ is a $\pi_{\Gamma\setminus G}^0$-tempered subgroup, for any irreducible lattice $\Gamma$. Here $\Delta(\SL_n(F))$ is the diagonal embedded copy of $\SL_n(F)$ into a product of $k$ copies of $\SL_n(F)$. 
 Indeed, in general $n(\SL_n(F))=n-1$, for any $n \ge 3$ and Theorem \ref{tensor power} applies.
   \end{enumerate}

\subsection{Universal pointwise bounds : Harish Chandra function} 
\subsubsection{Spectral estimates on Iwasawa groups}
Let us recall the discussion from \cite{GN1}, and begin with the following general definition of groups with an Iwasawa decomposition, which we call Iwasawa groups. Besides semisimple linear algebraic groups, note that this class includes for example algebraic groups of affine transformations of Euclidean spaces or nilpotent groups, with semisimple Levi component. 

\begin{Def}{\bf Groups with an Iwasawa decomposition.}
\begin{enumerate}\item 
An lcsc group $G$ has an {\it Iwasawa decomposition} \index{Iwasawa decomposition}
if it has two closed
amenable subgroups $K$ and $P$, with $K$ compact and $G=KP$.

\item The {\it Harish-Chandra $\Xi$-function} \index{Harish-Chandra function} associated with the 
Iwasawa decomposition $G=KP$ of the unimodular group $G$ 
is given by 
$$\Xi(g)=\int_K \delta_G^{-1/2}(gk)dk$$
where $\delta_G$ is the left modular function of 
$P$, extended to a left-$K$-invariant
function on $G=KP$. (Thus if $m_P$ is left Haar measure on $P$, 
$\delta_G(p) m_P$ is right $P$-invariant, and Haar measure on $G$ is given by 
$dm_G=dm_K \delta_G(p) dm_P$).
\end{enumerate}
\end{Def}

{\bf Convention}. The definition of an Iwawasa group 
involves a choice of a 
compact subgroup and an amenable 
subgroup. 
When $G$ is the the group of $F$-rational points 
of a semisimple algebraic 
group defined over a locally compact non-discrete field $F$, 
$G$ does admit an Iwasawa decomposition, and 
we can and will always choose below $K$ to be 
a good maximal compact subgroup,  
and $P$ a corresponding minimal 
$F$-parabolic group (see \cite{Marg} or \cite{Si}). 
  This choice 
will be naturally extended 
in the obvious way to $S$-algebraic groups.

Let now $G$ be an lcsc group, $
K$ a compact subgroup, and $\pi : G\to \mathcal{U}(\mathcal{H})$
be a strongly continuous unitary representation, 
where $\cU(\cH)$ is
the unitary group of the Hilbert space $\cH$. We recall the following basic spectral estimates for general 
Iwasawa groups, which were stated in \cite{GN1} and constitute a natural generalization of \cite{chh}.

\begin{Theorem}\cite{GN1}\label{Kfinite estimate}
Let $G=KP$ be a unimodular lcsc 
group with an Iwasawa decomposition,  
and $\pi$ a strongly continuous
unitary representation of $G$. Let 
$v$ and $w$ be two $K$-finite vectors, and denote the dimensions of their spans
under $K$ by $d_v$ and $d_w$. Then the following estimates hold, 
where $\Xi$ is the Harish-Chandra $\Xi$-function.
\begin{enumerate}
\item If $\pi$ is strongly $L^{2+\eta}$ for every $\eta > 0$, then $\pi$ is weakly contained in the regular representation of $G$. 
\item 
 If $\pi$  
is weakly contained in the regular representation, then 
$$\abs{\inn{\pi(g)v,w}}\le \sqrt{d_v d_w}
\norm{v}\norm{w}\Xi(g)\,.$$

\item If $\pi$ is strongly $L^{2k+\eta}$ for 
all $\eta > 0$, then 
$$\abs{\inn{\pi(g)v,w}}\le
 \sqrt{d_v d_w}\norm{v}\norm{w}\Xi(g)^{\frac1k}\,.$$

\end{enumerate}
\end{Theorem}

Let us note the following : 
\begin{enumerate}
\item The quality of the estimate in  
Theorem \ref{Kfinite estimate} varies with the structure of $G$.
 For example, if $P$ is normal in $G$ (so that  
 $G$ is itself amenable) then $P$ is unimodular if $G$ is. Then 
$\delta_G(g)=1$ for $g\in G$ 
 and the estimate is trivial. 
 \item For semisimple algebraic groups over $F$ the $\Xi$-function 
is in fact in $L^{2^+}$, a well-known result due to Harish-Chandra. 
\item Theorem \ref{Kfinite estimate} 
will be most useful when the Harish-Chandra
function is indeed in some $L^{p^+}(G)$, $p^+ < \infty$, 
so that Theorem 
\ref{tensor} applies. We note that in \cite{GN2}, it was established that the Harish Chandra function 
of the adele groups of semi simple algebraic groups have this property, with $p^+=4$. 

\item When $\Xi$ is in some $L^q$, $q < \infty$, Theorem 
\ref{Kfinite estimate}(1) implies that any  
representation a tensor power of which is 
weakly contained in the regular representation is strongly 
$L^p$ for some $p$.
\end{enumerate}

Theorem \ref{Kfinite estimate} is a universal pointwise bound for matrix coefficients of unitary representation of Iwasawa groups. When $G$ is a realsemisimple algebraic group the Harish Chandra $\Xi_G$ function admits an explicit  two-sided estimate in terms of  $\delta_G$, as follows. We denote by $\tilde{\delta}_G$ the $W$-invariant function extending $\delta_G$ from $A^+$ to all of $A$, where $W$ is the Weyl group of $A$ in $G$. Then 
\begin{equation}\label{eq: modular-funct}
B_1(G) \tilde{\delta}_G(g)^{-1/2} \le \Xi_G(g) \le B_2(G) \tilde{\delta}_G^{-1/2}(g)(1+\norm{g})^{d(G)}\end{equation}
where $\log \norm{g}=\text{dist} (gK,K)$, $\text{dist}$ the $G$-invariant Riemannian distance given by the Killing form on symmetric space $G/K$ (or a suitable $G$-invariant distance on the Bruhat-Tits building in the local field case). 
The modular function $\delta_G$ on $A$ is given by the product of all the characters associated with the roots of an $F$-split torus $A\subset P$, namely by the character associated with the sum of all positive roots of $A$ in $G$.
Finally, $d(G)$ is a constant depending only on $G$, and bounded by the number of roots. 
Defining $n(p)$ to be the least integer satisfying $n(p)\ge p/2$, and keeping our notation above, we can now state the following explicit integral criterion for subgroup temperedness. 

\begin{Prop}\label{pts-bounds}
Let $H$ be a closed subgroup of an Iwasawa group $G$, and $\pi$ an $L^{p^+}$-representation of $G$. 
\begin{enumerate}
\item  If $\Xi_G^{1/n(p)}$ restricted to $H$ is in $L^{2+\eta}$ for every $\eta > 0$, then $(G,H,\pi)$ is a tempered triple, 
i.e. $\pi$ restricted to $H$ is weakly contained in $\infty \cdot \lambda_H$.
\item Let $G$ be a semisimple algebraic group,  $H$ a semisimple algebraic $F$-subgroup, and choose the Cartan and Iwasawa decomposition of $H$ to be the restrictions of those chosen on $G$.
If the product  $\tilde{\delta}_G^{-1/n(p)}\cdot \delta_H$ is in $L^{2+\eta}(A_H^+)$ for every $\eta>0$ (w.r.t Haar measure on the positive Weyl chamber $A_H^+\subset A\cap H$), then $(G,H,\pi)$ is a tempered triple.  

\end{enumerate}
\end{Prop}
\begin{proof}

1) Using Theorem \ref{Kfinite estimate} and the definition of $n(p)$,  the assumption implies that the
$K$-finite matrix coefficients of $\pi$ are in $L^{2+\eta}(H)$ for every $\eta>0$ when restricted to $H$. By \cite{chh}, a representation which has a dense set of matrix coefficients which are in $L^{2+\eta}$ is weakly contained in a multiple of the regular representation. Finally, $K$-finite vectors are dense in any unitary representation. 
 
2) We need only verify that the restriction of a $K$-finite matrix coefficient of $G$ to $A_H^+$ is in $L^{2+\eta}$ w.r.t. the density arising on $A_H^+$ by writing Haar measure on $H$ in Cartan coordinates. The latter density is bounded on $A_H^+$ by $\delta_H^{2+\eta}$. Using Theorem \ref{Kfinite estimate} again, the matrix coefficients in question are bounded on $A^+$ by (for all $\eta > 0$)
$$\Xi_G^{1/n(p)}\le \left(B_G(1+\norm{\cdot})^{d(G)}\delta_G\right)^{-1/n(p)}\le B(G,p, \eta)\delta_G^{-1/n(p)+\eta}.$$
The function $\Xi_G$ is invariant under the Weyl group of $A$ in $G$, and thus when we extend the r.h.s. of the estimate above so as to be invariant under the Weyl group, we obtain an estimate of $\Xi_G$ on all of $A$. 
Using the matching Cartan decompositions of $H$ and $G$, we restrict this estimate to $A_H^+$, and conclude that integrability (w.r.t. Haar measure) of $\left(\tilde{\delta}_G^{-1/n(p)}\cdot\delta_H\right)^{2+\eta}$ on $A_H^+$ 
implies the (almost) square-integrability of the matrix coefficients restricted to $H$. 
\end{proof}

\subsection{Universal pointwise bounds using strongly orthogonal systems}

Another universal pointwise bound for $K$-finite matrix coefficients was developed by Howe \cite{H} and Howe and Tan \cite{HT}. To state it for $\SL_n(\R)$, let $\hat{a}=(a_1,\dots,a_n)$ parametrize the element of the diagonal group $A\subset \SL_n(\R)$, with $A^+$ being given by the condition $a_1\ge a_2,\dots, \ge a_n$. We let $\Xi_2$ denote the Harish Chandra $\Xi$-function of $\SL_2(\R)$, 
and define the following bi-$K$-invariant function on $\SL_n(\R)$ : 
\begin{equation}\label{HT-Psi}\Psi(k_1\hat{a}k_2)=\min_{i\neq j}\Xi_2\left(\begin{matrix}\sqrt{a_i/a_j}& 0\\ 0 & \sqrt{a_j/a_i}\end{matrix}\right)\end{equation}
Then the following universal pointwise bound holds \cite{HT}. For any two $K$-finite vectors $u$ and $v$ in a unitary representation $\pi$ of $\SL_n(\R)$ without invariant unit vectors,  
\begin{equation}\label{HT-estimate}\abs{\inn{\pi(\hat{a})u,v}}\le \sqrt{d_ud_v}\norm{u}\norm{v}\Psi(\hat{a})\,.\end{equation} 

The foregoing construction uses some naturally embedded copies of $\SL_2(\R)$ in $\SL_n(\R)$, namely 
$H_{i,j}$ stabilizing the plane spanned by the standard basis vectors $e_i,e_j$ and acting trivially on the other basis vectors. It then applies the fact that the unitary representations of $\SL_2(\R)\rsemi \R^2$ are tempered when restricted to $\SL_2(\R)$, provided the representation has no $\R^2$-invariant unit vectors. 

More generally, it is possible to use a family of commuting copies of $\SL_2(\R)$ in a semisimple group $G$ in order to bound the matrix coefficient by an expression involving the Harish Chandra functions of the $\SL_2(\R)$ subgroups. A commuting family of such subgroup arises from any strongly  orthogonal root system $\Upsilon=\set{\beta_1,\dots,\beta_k}$ in the root system of $G$, namely a system satisfying that $\beta_i\pm \beta_j$ is not a root for $1\le i\neq j\le k$. This construction was elaborated  systematically for all simple algebraic groups over local fields in \cite{oh2}. It produces, for each such $G$, and for each strongly orthogonal root system $\Upsilon$, a bi-$K$-invariant function $\xi_{G,\Upsilon}$ satisfying \cite[Thm. 1.1]{oh2}: 
$$\abs{\inn{\pi(g)v,w}}\le  B_G \sqrt{d_v d_w} \norm{v}\norm{w} \xi_{G,\Upsilon}(g)$$
for every $K$-finite matrix coefficient as above. The function $\xi_{G,\Upsilon}$ in question 
is given on $A^+$ by a product of the $\Xi_2$-functions associated with the collection of subgroups $H_\beta$, $\beta\in \Upsilon$, each of which is isomorphic to $\SL_2(F)$ or $P\SL_2(F)$. 

Thus, when the restriction of $\xi_{G,\Upsilon}$ to a closed subgroup $H\subset G$ is in $L^{2+\eta}(H)$
for every $\eta>0$, the subgroup $H$ is tempered in $G$. When $H$ is a semisimple  algebraic
$F$-subgroup, and we choose the Cartan and Iwasawa decompositions of $H$ to be the restrictions of those
chosen on $G$, the integrability condition amounts to $\xi_{G,\Upsilon}\cdot \delta_H$  restricted to
$A^+_H$ being in $L^{2+\eta}(H)$  for every $\eta>0$ w.r.t. Haar measure on $A_H$. Again this condition can be phrased in terms of the characters associated with the strongly orthogonal root system when restricted to $A^+_H$, and is a different condition than the one provided by the restriction of the $\Xi_G^{1/n(p)}$-function to $A^+_H$.

In particular if $G(F)$ be simple, let $H\subset G$ be the subgroup $H=\prod_{\beta\in \Upsilon} H_\beta$  associated with a strongly orthogonal system $\Upsilon$ of roots in the root system of $G$ over $F$. Then $H$ is a tempered subgroup of $G$, as follows from \cite{HT}\cite{H} for the real and complex case, and \cite{oh2} more generally.

\subsection{Subgroups which are $(G, K)$ tempered in $L^1$ are tempered}

 In \cite{Ma} and \cite{oh2} several examples were given of subgroups $H\subset G$ which are $(G,K)$-tempered in $L^1$, where both $G$ and $H$ are semisimple. As we shall now prove, all of these examples are of tempered subgroups in the sense of the present paper. In many cases, both facts follow from a straightforward application of a suitable universal pointwise bound of $K$-finite matrix coefficients of $G$ restricted to $H$.
 We note that the condition that all $K$-invariant matrix coefficients of irreducible unitary representations of $G$ being in $L^{2+\eta}(H)$ for all $\eta > 0$ is of course  weaker that being in $L^1(H)$. It is strictly weaker, since for  example,   the embedding of 
$\SL_2(\R)\subset \SL_n(\R)$ in the upper left corner is not $(G,K)$-tempered in $L^1$ (see e.g. example
5.3 in \cite{oh2}), but it is tempered in the usual sense by the basic Kazhdan argument. Thus the
$K$-invariant matrix coefficients in this case are $L^{2+\eta}(H)$ for every $\eta>0$, but not in $L^1(H)$. In general we have 

\begin{Theorem}\label{thm:L1L2}  Let $H$ be a closed subgroup of the semisimple group $G$. If $H$ is 
$(G,K)$-tempered in $L^1$, then $H$ is a tempered subgroup of $G$.
\end{Theorem}
\begin{proof}
$K$-finite vectors are dense in every continuous unitary representation $G$ on a separable Hilbert space. It suffices to show given a unitary representation $\pi$ of $G$ without invariant unit vectors, that the restriction of $K$-finite matrix coefficients of $\pi$ to $H$ are in $L^{2+\eta}(H)$ for every $\eta > 0$. We will use 
an extension of an argument of Cowling \cite[Lem. 2.2.6]{Co}, and prove the following more general fact.  

 Assume that for every unitary representation $\pi$ of $G$ without invariant unit vectors, the restriction of the $K$-invariant matrix coefficients of $\pi$ to $H$ are in $L^p(H)$. We claim that then the restriction of all $K$-finite matrix coefficients of every such unitary representations $\pi$ to $H$ is in $L^{2p}(H)$.  
 
Given this claim, the condition that $H$ is $(G,K)$-tempered in $L^1$ (and even the weaker condition that
it is $(G,K)$-tempered in $L^{1+\eta}$ for every $\eta>0$) implies that the $K$-finite matrix coefficients of $G$ restricted to $H$ are 
in $L^{2+\eta}(H)$ for every $\eta>0$, namely almost square integrable. By Theorem \ref{tensor} $\pi$ restricted to $H$ is weakly contained in $\infty \cdot \lambda_H$ and $H$ is a tempered subgroup of $G$. 

To prove the claim let $u$ and $v$ in $\mathcal{H}_\pi$ be $K$-finite, and then there exist trigonometric polynomials $r_1,\dots ,r_N$ on $K$, such that for $w=u,v$ and all $k\in K$ 
$$ \pi(k)w=\sum_{j=1}^N \bar{r}_j(k)w_j.$$
The linear span (under left translation) of the functions $\bar{r}_j(k^\prime) \cdot r_i(k)$, $1\le i,j\le N$  in $C(K_H\times K_H)$ and $C(K\times K)$ is finite dimensional, and hence the restriction of the $L^2$-norm, the $L^\infty$-norm and the $L^p$ norm to the subspace are all equivalent, namely 
$ \norm{\cdot}_{L^{2p}(K_H\times K_H)}\le C_{2p,\infty} \norm{\cdot}_{L^\infty(K_H\times K_H)}$, 
and 
$ \norm{\cdot}_{L^{\infty}(K\times K)}\le C_{\infty,2} \norm{\cdot}_{L^2(K\times K)}$. 

Therefore 
\begin{align*}
&\norm{\inn{\pi(\cdot)u,v}}_{L^{2p}(H)}^{2p}\\
&=
\int_H \int_{K_H}\int_{K_H} \abs{\inn{\pi(khk^\prime)u,v}}^{2p} dk dk^\prime dh \\
&=\int_H \int_{K_H}\int_{K_H} \abs{\sum_{j,j^\prime=1}^N r_{j^\prime}(k^\prime) r_j(k)\inn{\pi(h)u_j,v_{j^\prime}}}^{2p} dk dk^\prime dh \\
&\le C_{2p,\infty}^{2p}\int_H \max_{(k,k^\prime)\in K_H\times K_H} \abs{\sum_{j,j^\prime=1}^N r_{j^\prime}(k^\prime) r_j(k)\inn{\pi(h)u_j,v_{j^\prime}}}^{2p} dh \\
&\le C_{2p,\infty}^{2p}\int_H \max_{(k,k^\prime)\in K\times K} \abs{\sum_{j,j^\prime=1}^N r_{j^\prime}(k^\prime) r_j(k)\inn{\pi(h)u_j,v_{j^\prime}}}^{2p} dh \\
&\le C_{\infty, 2}^{2p}C_{2p,\infty}^{2p} \int_H \abs{\int_{K}\int_{K} \abs{\sum_{j,j^\prime=1}^N r_{j^\prime}(k^\prime) r_j(k)\inn{\pi(h)u_j,v_{j^\prime}}}^{2}  dk dk^\prime}^p dh\\
&=  C_{\infty, 2}^{2p}C_{2p,\infty}^{2p} \int_H \abs{\int_{K}\int_{K} \abs{\inn{\pi(khk^\prime)u,v}}^{2}  dk dk^\prime}^p dh
\\
&=  C_{\infty, 2}^{2p}C_{2p,\infty}^{2p} \int_H \abs{\int_{K}\int_{K}\inn{(\pi\otimes
    \pi^c)(khk^\prime)(u\otimes u^c,v\otimes v^c}  dk dk^\prime}^p dh\\
&=  C_{\infty, 2}^{2p}C_{2p,\infty}^{2p} \int_H \abs{\inn{\pi\otimes\pi^c(h)U,V} }^p dh
\end{align*}
with $\pi^c$ being the  representation contragredient to $\pi$, namely $\inn{\pi^c u^c,v^c}=\overline{\inn{\pi(g)u,v} }$, and 
$$U=\int_K\pi(k^\prime)(u\otimes u^c) dk^\prime, \text{ , } V=\int_K\pi(k^{-1})(v\otimes v^c) dk\,.$$
Clearly $U$ and $V$ are bi-$K$-invariant vectors, and the matrix coefficient they determine is an $L^p$-function on $H$ by our assumption, provided only that the representation $\pi\otimes \pi^c$ does not admit $G$-invariant unit vectors. But since both 
$\pi$ and $\pi^c$ do not admit $G$-invariant unit vectors, by the Howe-Moore theorem the matrix coefficients of $\pi\otimes \pi^c$ vanish at infinity on $G$. This completes the proof of Theorem \ref{thm:L1L2}. 
\end{proof}

Let us now note some further examples of tempered subgroups $H\subset \SL_n(\R)$ based on the foregoing discussion. 
\begin{enumerate}
\item According to \cite[Ex. b]{Ma},  the irreducible representation of $\SL_2(\R)$ into $\SL_n(\R)$, $n \ge 4$ yields a subgroup which is $(G,K)$-tempered in $L^1$, and hence a tempered subgroup. 
As noted above, the irreducible representation in the case  $n=3$ yields a tempered subgroup by Kazhdan's argument, which is not $(G,K)$-tempered in $L^1$. 
\item According to \cite[Ex. c]{Ma} given any simple group $H$, there exists $k=k(H)$ such that for any irreducible representation of $\tau : H\to \SL_n(R)$, $\tau(H)$ is $(G,K)$-tempered in $L^1$ and hence a tempered subgroup of $\SL_n(R)$ provided $n \ge k$.  
\item Even more generally, according to \cite[Ex. c]{Ma}, for any (possibly reducible) representation $\tau : H \to \SL_n(\R)$, if the dimension $n'$ of the sum of the irreducible non-trivial subrepresentations of $\tau$  satisfies $n'\ge k'$ (with $k'$ a suitable constant), then again $\tau(H)$ is $(G,K)$-tempered in $L^1$ and hence tempered. 
\item If $L$ is any simple subgroup of dimension $n$, then $Ad(L)\subset \SL_n(\R)$ is a tempered subgroup of $\SL_n(R)$. Indeed, according to \cite[Ex. \S 5.1]{oh2}, the restriction of the universal pointwise bound 
$\xi_{\Upsilon}$ for a suitable maximal strongly orthogonal system $\Upsilon$  to the subgroup $Ad (L)$ is just $\delta_L$. 
Since $\delta_L$ is in $L^{2+\eta}(A_L^+)$ for every $\eta>0$, it follows that  $Ad(L)$ is a tempered subgroup of $\SL_n(\R)$. 
\end{enumerate}

\subsection{Bounds towards the Selberg eigenvalue conjecture.} 
%

For simple real groups of real rank one, bounds towards the Selberg eigenvalue conjecture for certain pairs $(G, \Gamma)$ have been developed, and in many cases they are optimal or close to optimal. It is possible to  exploit such bounds to establish best possible (or close to best possible) exponents of Diophantine approximation in number of very interesting and natural examples. We will consider this topic in detail in \cite{GGN4}, and here we will just briefly mention some of the possibilities.

First, if the representation of $G$ on $L^2_0(\Gamma \setminus G)$ is tempered, then so is its restriction to any closed subgroup $H\subset G$.  The main interest in this fact lies in the cases where the subgroup $H$ is non-amenable, and thus the case of $G=\SL_2(\R)$ is excluded here. We list some possibilities 
\begin{enumerate}
\item $G=\SL_2(\C)$, $H=\SL_2(\R)$, $\Gamma$ is an arithmetic lattice as well as some of its low level congruence groups, for example $\Gamma=\SL_2(\Z[i])$. 
\item Let $G$ be a simple Kazhdan group of real rank one, so that $G=Sp(n,1)$ or the exceptional group $F_4^{-20}$.  There is an embedding of $H=Sp(2,1)\subset Sp(n,1)=G$ (whenever $n \ge 3$)  which gives a tempered subgroup (see \cite[Sec. 4.8]{oh2}), and thus so is also every closed subgroup  of $H$. 
Note that the temperedness here applies to all the representations $\pi^0_{\Gamma\setminus G}$ for all  lattices $\Gamma \subset Sp(n,1)$.

\item The representation of $G=\SL_2(\R)\times \SL_2(\Q_p)$ on $L^2_0(\Gamma\setminus G)$ where $\Gamma=\SL_2(\Z[\frac1p])$ is tempered, provided $p$ is such that the congruence subgroup $\Gamma(p)\subset \SL_2(\Z)$ satisfies the Selberg eigenvalue conjecture. This is indeed the case for $p \le P$ for some $P$. Thus again every subgroup of $G$ is tempered in this representation. 
\end{enumerate}

Second, spectral estimate weaker than temperedness have been established for arithmetic lattices in $\SO(n,1)$ or $\SU(n,1)$. Such eignenvalue bounds are sufficient to imply that the restriction of $\pi^0_{\Gamma\setminus G}$ to suitable $\SO(k,1)$ or $\SU(k,1)$ is tempered, or close to tempered. They 
allow the derivation of Diophantine exponents which are optimal or close to optimal in the case of the action of arithmetic lattice on rational ellipsoids, for example. For more details we refer to \cite{GGN4}.

\section{The examples, revisited}\label{sec:proofs}\label{sec:ex}
We now turn to complete the proofs of the corollaries stated in Section \ref{sec:examples}. In every
case, we will first identify the algebraic group $G$, its algebraic subgroup $H$ (and thus the variety
$G/H$), the distance function on $H$, and the lattice subgroup $\Gamma$. We will then show that the
representation of $H$ in $L^2(\Gamma\setminus G)$ is tempered, identify the rate of volume growth  on $H$
with respect to the distance function, and then use Theorem
\ref{thm:optimal}  to conclude that the rate of Diophantine approximation is the one stated in the corollary. 

\subsection{Proof Corollary \ref{ex:1}}{\it Affine action on the real plane.} Here $G=\SL_2(\R)\rsemi\R^2$, acting affinely on the plane $\R^2$, $H=\SL_2(\R)$ is the stability group of the zero vector, so that $G/H=\R^2$. 
The lattice is $\Gamma=\SL_2(\Z)\rsemi \Z^2$, namely the lattice of integral points. We take any norm on
$\R^2$, and  any norm on $H$ induced from a matrix norm on $\Matn_3(\R)$. It  is  well-known that the volume growth on $H$ is  given by 
$\vol(\set{h \in H\,;\, \norm{h} < T})\sim C_{\norm{\cdot}} T^2$ as $T\to\infty$.  As to the spectral estimate, the $G$-equivariant  factor map $\Gamma\setminus G\to \SL_2(\Z)\setminus \SL_2(\R)$ allows us to lift $L^2(\SL_2(\Z)\setminus \SL_2(\R))$ and embed it as an 
$G$-invariant subspace $\cH$ of $L^2(\Gamma\setminus G)$. It is well-known that the representation of $H=\SL_2(\R)$ on $\mathcal{H}\cong L^2(\SL_2(\Z)\setminus \SL_2(\R))$ is tempered. We claim that the representation of $H$ on $\mathcal{H}^\perp\subset L^2(\Gamma\setminus G)$ is also tempered. 
Indeed the unitary representation of $G$ on $\cH^\perp$ is such that the subgroup $\R^2$ does not have an
invariant unit vector. As a result, Kazhdan's argument as formulated in Theorem \ref{thm:sl2} implies
that the representation of $H=\SL_2(\R)$ on $\cH^\perp$ is tempered. Thus here $a=2$, $d=\dim G/H=2$, and
we conclude using Theorem  \ref{thm:optimal} that  the Diophantine exponent is $\kappa=d/a=1$. 
 
 \subsection{Proof of Corollary \ref{ex:2}}{\it Affine action on the complex plane.} Here
 $G=\SL_2(\C)\rsemi \C^2$, $G$ acts affinely  on $G/H=\C^2$, with stability group
 $H=\hbox{Stab}_G(0)=\SL_2(\C)$, and we consider the lattice $\Gamma=\SL_2(\Z[i])\rsemi \Z[i]^2$. We take
 any norm on $\C^2$, and any (vector space) norm on $\Matn_2(\C)$ restricted to $H$, and in that case the
 volume growth is  $\vol(\set{h \in H\,;\, \norm{h} < T})\sim C_{\norm{\cdot}} T^4$ as $T\to\infty$. Indeed, in general the volume growth of balls defined with respect to any (vector space) norm on $\SL_n(F)\subset \Matn_n(F)$ is $T^{n^2-n}$ when $F=\R$ according to \cite{DRS}, and $T^{2(n^2-n)}$ when $F=\C$. Again, the representation of $H=\SL_2(\C)$ on $\cH=L^2(\SL_2(\Z[i])\setminus \SL_2(\C))$ is known to be tempered (see below), and the representation of $G$ on the orthogonal complement $\cH^\perp$ is such that $\C^2$ does not have invariant unit vectors.  Thus Theorem \ref{thm:sl2} implies that the representation of $H=\SL_2(\C)$ on $\cH^\perp$ is tempered. Clearly here $d=4$ and since $a=4$ it follows  that $\kappa_\Gamma(\C^2)=1$ according to Theorem \ref{thm:optimal}.  

Consider now the lattice subgroup $\Gamma=\SL_2(\cI)\rsemi \cI^2$, where $\cI$ is the ring of Eisenstein integers contained in  
$\Q[\sqrt{-3}])$, or the rings of integers of the imaginary quadratic fields  $\Q[\sqrt{-2}])$ and
$\Q[\sqrt{-7}])$.  it has been established directly that the Laplace operator acting on the manifolds
$\SL_2(\cI)\setminus \SL_2(\C)/\SU_2(\C)$ has no exceptional eigenvalues, see \cite[Prop. 6.2,
p. 347]{EGM}.  It follows that the spherical spectrum of the representation of $\SL_2(\C)$ in
$L^2(\SL_2(\cI)\setminus \SL_2(\C))$ is tempered. Now, if we take a  norm on $H$
biinvariant under $K=\SU_2(\C)$, then the norm balls would be bi-$K$-invariant sets, and the spectral
bound on the rate of decay of the associated averaging operators in the representation space is
determined by the spherical spectrum. Hence the operator norms decay in this case is indeed the best
possible, given by the reciprocal of the square root of the volume.
Hence, $\kappa_\Gamma(\C^2)=1$ as above.
Finally, this estimate  clearly holds for any norm on $H$ if it holds for one choice. 

We note that the  case of a general Bianchi group $\SL_2(\cI)$ for $\cI$ the ring of integer of an imaginary quadratic field $\Q[\sqrt{-D}])$ ($D\ge 2$ a square-free positive integer), Selberg conjecture is still open even for the lattice of integral points (i.e. level one congruence group), and so is the problem of best possible Diophantine approximation.

\subsection{Proof of Corollary \ref{ex:3}} {\it Ternary forms.} 
The variety $\mathcal{Q}_{\sigma,d}(\R)$
can thus be identified with the set
of $3\times 3$ symmetric matrices with
determinant $d$ and signature $\sigma$.
The group $G=\SL_3(\R)$ acts on this variety  via 
$(g,X)\mapsto {}^tgXg$, and the stability group is isomorphic to $H=\SO(2,1)$. 
Thus $\mathcal{Q}_{\sigma,d}(\R)$ can be identified with the homogeneous space $G/H$. 
We note that $\SO(2,1)\cong \PSL_2(\R)$, but the representation of $\SO(2,1)$ in $\SL_3(\R)$ is
irreducible and Kazhdan's argument does not apply. Rather, we use the fact that given any irreducible
representation $\sigma_n : \SL_2(\R)\to\SL_n(\R)$ the image $H_n=\sigma_n(\SL_2(\R))$ is a tempered
subgroup of $\SL_n(\R)$ provided $n \ge 3$, as follows from Theorem \ref{thm:L1L2}. Given any norm on
$\Matn_n(\R)$, the volume growth of norm balls in $H_n=\sigma_n(\SL_2(\R))$ is asymptotic to
$C_{\norm{\cdot},n}T^{2/(n-1)}$ (see, for instance, \cite{DRS} or \cite[Sec.~7]{GW}). In the case of ternary forms $n=3$ so that  $a=1$, and $d=\dim G/H=8-3=5$, and we conclude that the exponent of Diophantine approximation is $\kappa_\Gamma(\mathcal{Q}_{\sigma,d}(\R))=5$ for {\it any} lattice subgroup $\Gamma\subset \SL_3(\R)$.

\subsection{Proof of Corollary \ref{ex:4}} {\it Constant determinant variety.} 
Here $G=\SL_3(F)\times \SL_3(F)$, $H=\Delta( \SL_3(F))$ the diagonally embedded copy of $\SL_3(F)$, and
$G/H\cong \SL_3(F)$. We take distances given by a norm on $\Matn_3(F)$ restricted to $H$ and $G/H$. As is
well known (see \cite{li} and \cite{oh2}) here $n(\SL_3(F))=4$, and so every irreducible infinite
dimensional representation of $\SL_3(F)$ is  $L^{4+\eta}$ for every $\eta>0$. Hence the restriction of
the tensor power of two such representations to the diagonal subgroup $\Delta(\SL_3(F))$ is in
$L^{2+\eta}$ for every $\eta>0$, as shown in Theorem \ref{tensor power}. Thus, while $H$ is not a tempered subgroup of $G$, for every {\it irreducible} lattice $\Gamma\subset \SL_3(F)\times \SL_3(F)$, the representation of $H$ in $L^2_0(\Gamma\setminus G)$ is tempered. The rate of volume growth of norm balls in $\SL_3(F)$ is asymptotic to $T^6$ when $F=\R$ or $F=\Q_p$ and $T^{12}$ when $F=\C$. The (relevant) dimension of $G/H$ is $d=8$ when $F=\R,\Q_p$, and $d=16$ when $F=\C$. Thus for any irreducible lattice, $\kappa_\Gamma(G/H) =4/3$. 

\subsection{Proof of Corollary \ref{ex:5}} {\it Complex structures.}
Here $G=\SL_4(\R)$, $H=\SL_2(\C)$, and $G/H$ is the variety of complex structures on $\R^4$. We claim that for any lattice $\Gamma$ of $\SL_4(\R)$, the restriction of the representation of $\SL_2(\C)$ on $L^2_0(\Gamma\setminus \SL_4(\R))$ is tempered. This follows from a straightforward explicit calculation using the natural embedding of $SL_2(\C)$ in $SL_4(\R)$ and applying  the universal pointwise estimate (\ref{HT-Psi}) given by Howe and Tan (see \S $4.6$). The real dimension of the variety $G/H$ is $d=9$, and the volume growth in $\SL_2(\C)$ in the defining representation is $a=4$.  We conclude that $\kappa_\Gamma(G/H)=9/4$. 

\subsection{Proof of Corollary \ref{ex:6}}{\it Simultaneous Diophantine approximation.}
Here $G=\SL_{2n}(F)$, $H=\SL_2(F)^{n} \subset \SL_{2n}(F)$, and the variety $\mathcal{D}_n(F)\simeq\SL_{2n}(F)/\SL_2(F)^n$ can be interpreted as the variety of "unimodular direct sum decompositions" of $F^{2n}$, to $n$ two-dimensional subspaces with suitable volume forms. We take a norm $\Matn_{2n}(F)$ and restrict it to $H$, and an embedding $G/H\subset F^k$ with distance given as a restriction of a norm on $F^k$. $H$ is a direct product group, and each $\SL_2(F)$-component of $H$ is tempered in $G$ by Kazhdan's argument, as stated in Theorem \ref{thm:sl2}.  Hence the same is true for the direct product group, namely $H$. The dimension of $G/H$ is clearly given by $(2n)^2-1-3\cdot n$ when $F=\R, \Q_p$, and when $F=\C$ the real dimension is doubled.  The exponent of volume growth of $H=\SL_2(F)^n$ is $n$ times the volume growth of $\SL_2(F)$ in the defining representation, namely $2n$ for $F=\R,\Q_p$ and $2n$ for $F=\C$.  We note however that the volume growth of norm balls in this case is in fact asymptotic to a multiple of $ T^{2n}$ (or $ T^{4n}$), so the exponent of volume growth is as stated. 
Thus for $F=\R,\Q_p,\C$, we have $\kappa_\Gamma(\mathcal{D}_n(F)) = ((2n)^2-1-3n)/2n$.

\subsection{Proof of Corollary \ref{ex:7}} {\it  Irreducible representation of $\SL_2$.}  
As already noted in \S 4.7, it was shown in \cite[Ex. b]{Ma} that for every irreducible representation $\sigma_n : \SL_2(\R)\to \SL_n(\R)$, $n \ge 3$, the subgroup $\sigma_n(\SL_2(\R))=H$ is a subgroup which is $(G,K)$-tempered in $L^1$. By Theorem 
\ref{thm:L1L2}, it follows that it is therefore tempered. We remark that the latter fact can also easily be proved directly using the Howe-Tan universal pointwise bound stated in (\ref{HT-Psi}) and (\ref{HT-estimate}). 
Taking any norm on $\Matn_n(\R)$, the growth of norm balls of radius $T$ restricted to $\sigma_n(H)$ is 
asymptotic to $T^{2/(n-1)}$ (see \cite{DRS} or \cite[Sec.~7]{GW}), and the dimension of the variety in question is of coarse $d=n^2-4$. The Diophantine exponent is given by $\kappa=d/a=\frac12(n^2-4)(n-1)$, by Theorem \ref{thm:optimal}. 

\subsection{Proof of Corollary \ref{ex:8}} {\it  Reducible  representations of $\SL_2$. }  
Let $\sigma : \SL_2(\R) \to \SL_n(\R)=G$ be a non-trivial representation, where $n \ge 3$, and denote $H=\sigma(\SL_2(\R))$. If $\sigma$ is irreducible, $H$ is a tempered subgroup of $G$ by Corollary \ref{ex:7}. 
Otherwise, $\sigma$ is reducible, and there is at least one $H$-invariant subspace $V\subset \R^n$, on which the action of $H$ is by an irreducible  representation of dimension $k$, where $2 \le k < n$. 
It follows that $H$ is contained in a subgroup $L$ of $G$, which is isomorphic to $\SL_2(\R)\rsemi_{\sigma_k} \R^k $. By Theorem \ref{thm:sl2 restriction}, $L$ is tempered subgroup of $G$ and hence so is $H$.
Choose a norm $\norm{\cdot}$ on $\Matn_n(\R)$, and let $a(\sigma, \norm{\cdot})$ denotes the rate of growth of norm balls in $H$ with respect to the norm restricted to $H$. Then for any lattice $\Gamma\subset \SL_n(\R)$, the exponent of Diophantine approximation on $G/H$ is given by $(n^2-4)/a(\sigma,\norm{\cdot})$, by Theorem \ref{thm:optimal}.  

\subsection{Proof of Corollary \ref{ex:9}}{\it Restriction of scalars for totally real fields.} 
$\SL_n(\R)$, $n \ge 3$ has the property that every unitary representation without invariant unit vectors is in $L^{2(n-1)+\eta}$. 
By Theorem \ref{tensor power}, the diagonal embedding $H=\Delta(\SL_n(\R))$ of $\SL_n(\R)$ in $G=\SL_n(\R)^{n-1}$ has the property that the restriction of the representation of $G$ on $L^2_0(G/\Gamma)$ to $H$ is tempered, for any irreducible lattice $\Gamma\subset G$.  The lattice $\Gamma$ associated with 
restriction of scalar from a totally real field of degree $n-1$ over $\Q$ is well known to be irreducible (see e.g. \cite{Marg}). Restricting the trace norm on $\SL_n(\R)^{n-1}$ to the diagonal subgroup, the volume growth of norm balls of radius $T$ is asymptotic to $T^{n^2-n}$. The dimension of the variety if of course $(n-2)(n^2-1)$, and thus by Theorem \ref{thm:optimal} we have $\kappa=(n+1)(n-2)/n$. 

\subsection{Covering homogeneous spaces.}
Finally, note the obvious fact that if a unitary representation $\pi$ of a unimodular lcsc group $H$ is weakly contained in a multiple of the regular representation of $H$, then for every closed unimodular subgroup $L\subset H$, the restriction of $\pi$ to $L$ is weakly contained in the regular representation of $L$. Thus if 
  $H$ is tempered in $G$, so is any subgroup $L\subset H$, and similarly, if the triple $(G,H,\pi)$ is tempered, then so is the triple $(G,L,\pi)$. When $L$ is a non-compact semisimple subgroup of the algebraic group $G$, temperedness implies that $\pi$ restricted to $L$ obeys best possible estimate of operator norms. Thus 
if $\Gamma$ has best possible rate on $G/H$, it has best possible rate also on its covers $G/L$, by Theorem \ref{thm:optimal}. However, in the computation of the exponent $\kappa=d/a$, the dimension $d$ of the homogeneous variety $G/L$ is larger than $\dim G/H$, and the rate of growth of norm balls  in $L$ is generally different than the rate of growth of norm balls in $H$.

\ignore{

\section{Further topics}
In the present paper we have emphasized the development of methods  establishing best possible Diophantine approximation rates for dense lattice orbits on homogeneous spaces, and we have focused mainly on examples arising when either the group $G$ or the stability group $H$ (or both) are  semisimple algebraic groups defined over fields of characteristic zero. Establishing upper bounds on the approximation exponents is possible is many other interesting problems, even if the bounds are not optimal. We mention the following problems  
\begin{enumerate}
\item An amenable stability group $H$, for example $\SL_2(\Z)$ acting linearly on the plane $\R^2=\SL_2(\R)/N$, $N$ unipotent. 
\item A non-unimodular stability group $H$, For example $\SL_n(\Z)$ acting on flag varieties $\SL_n(\R)/Q$, $Q$ a parabolic subgroup.
\item Homogeneous spaces of general Iwasawa groups $G$, for example the groups of affine transformations of vectors spaces and Heisenberg groups. 
\item Uniform estimates  for all congruence subgroups of an arithmetic lattice, the action of congruence subgroups on a rational hyperboloid of a given dimension.   
\item Algebraic groups defined over fields of positive characteristic. 
\item Using the power of log to improve the upper estimate. 
\item An analog of Khinchine's theorem for approximation on tempered homogeneous varieties. 
\end{enumerate}
These problems will be considered, in due course, in \cite{GGN4}. 

}

\end{document}